\documentclass[psamsfonts]{amsart}
\usepackage{cite}
\usepackage{amsfonts}
\usepackage{amscd}
\usepackage{amsthm}
\usepackage[all]{xy}
\usepackage{bm}
\usepackage[dvips]{graphicx}
\usepackage{graphicx}
\usepackage{verbatim}
\usepackage{enumerate}
\usepackage{amsmath}
\usepackage{amssymb}
\usepackage{url}
\usepackage{xcolor}
\usepackage{bbm}
\usepackage{mathabx}
\usepackage{mdframed}
\usepackage{color, soul}

\usepackage{appendix}

\allowdisplaybreaks

\newcommand{\RR}{\mathbb{R}}

\newcommand{\DD}{\mathbb{D}}

\newcommand{\CC}{\mathbb{C}}
\newcommand{\NN}{\mathbb{N}}
\newcommand{\SSS}{\mathbb{S}}
\newcommand{\CP}{\mathcal{P}}
\newcommand{\CK}{\mathcal{K}}
\newcommand{\TT}{\mathbb{T}}
\newcommand{\ZZ}{\mathbb{Z}}

\newcommand{\ve}{\varepsilon}
\newcommand{\CB}{\mathcal{B}}

\newcommand{\CDD}{\mathcal{D}}

\newcommand{\CF}{\mathcal{F}}
\newcommand{\CA}{\mathcal{A}}

\newcommand{\CU}{\mathcal{U}}

\newcommand{\CL}{\mathcal{L}}

\newcommand{\CN}{\mathcal{N}}

\newcommand{\La}{\Lambda}
\newcommand{\la}{\lambda}

\newcommand{\Si}{\Sigma}
\newcommand{\Om}{\Omega}
\newcommand{\om}{\omega}
\newcommand{\de}{\delta}
\newcommand{\Ga}{\Gamma}
\newcommand{\ga}{\gamma}
\newcommand{\al}{\alpha}

\newcommand{\Jac}{\mathrm{Jac}}

\newcommand{\p}{\partial}

\newcommand{\Id}{\mathrm{Id}}

\newcommand{\Sp}{\mathrm{Sp}}

\newcommand{\hF}{\widehat{F}}

\newcommand{\wht}{\widehat}
\newcommand{\lan}{\langle}
\newcommand{\ran}{\rangle}

\newcommand{\hb}{\hbar}
\newcommand{\bi}{\mathbf{i}}

\newcommand{\bzero}{\mathbf{0}}

\newcommand{\CW}{\mathcal{W}}
\newcommand{\wF}{\widehat{F}}
\newcommand{\wtF}{\widetilde{F}}

\newcommand{\Spec}{\mathrm{Spec}}
\newcommand{\opsi}{\overline{\psi}}

\newcommand{\bn}{\mathbf{n}}
\newcommand{\OP}{\mathrm{OP}}
\newcommand{\Op}{\mathrm{Op}}
\newcommand{\wP}{\widetilde{P}}
\newcommand{\wtp}{\widetilde{p}}
\newcommand{\wQ}{\widetilde{Q}}
\newcommand{\oa}{\overline{a}}
\newcommand{\wW}{\widetilde{W}}

\newcommand{\wa}{\widetilde{a}}

\newcommand{\wtr}{\widetilde{r}}

\newcommand{\ldb}{\ldbrack}
\newcommand{\rdb}{\rdbrack}

\newcommand{\lbnur}{\ldbrack{\bm{\nu}}\rdbrack}
\newcommand{\bnu}{{\bm{\nu}}}

\newcommand{\fB}{\mathfrak{B}}

\newcommand{\CT}{\mathcal{T}}

\newcommand{\mmod}{\mathrm{mod}}
\newcommand{\Const}{\mathrm{Const}}

\newcommand{\SU}{\mathrm{SU}}

\newtheorem{theoremmain}{Theorem}

\newtheorem{theorem}{Theorem}[section]
\newtheorem{prop}[theorem]{Proposition}
\newtheorem{lemma}[theorem]{Lemma}
\newtheorem{sublemma}[theorem]{Sublemma}
\newtheorem{cor}[theorem]{Corollary}

\newtheorem{remark}[theorem]{Remark}
\newtheorem{definition}[theorem]{Definition}

\numberwithin{equation}{section}

\newtheorem{claim}{Claim}

\def\disp{\displaystyle}
\def\id{\mathop{\hbox{\rm id}}}

\title[]
{The Spectral Gap for Transfer Operators of Torus Extensions over Expanding Maps}

\author{Jianyu Chen}
\address{Department of Mathematics and Statistics, 
University of Massachusetts Amherst,
Amherst, MA 01003, USA}
\email{jchen@math.umass.edu}

\author{Huyi Hu}
\address{Department of Mathematics, Michigan State University, 
East Lansing, MI 48824, USA}
\email{hu@math.msu.edu}

\begin{document}

\begin{abstract} 
We study the spectral gap for transfer operators
of the skew product 
$F: \TT^d\times \TT^\ell\to \TT^d\times \TT^\ell$ given by 
$F(x,y)=(Tx, y+\tau(x) \pmod{\ZZ^\ell})$, 
where $T: \TT^d\to \TT^d$ is a $C^\infty$ uniformly expanding endomorphism, 
and the fiber map $\tau: \TT^d\to \RR^\ell$ is a $C^\infty$ map. 
We construct a Hilbert space $\CW^{-s}$ for any $s<0$,
which contains  
all the H\"older functions of H\"older exponents $|s|$ on $\TT^d\times \TT^\ell$. 
Applying the method of semiclassical analysis, we obtain the dichotomy: 
either the transfer operator has a spectral gap  on $\CW^{-s}$,
or $\tau$ is an essential coboundary.
In the former case, $F$ mixes exponentially fast for H\"older 
observables with H\"older exponents $|s|$; 
and in the latter case, either
$F$ is not weak mixing and it is semiconjugate to a circle rotation, 
or $F$ is unstably mixing, i.e., it can be approximated by non-mixing skew products.
\end{abstract}

\maketitle


\section{Introduction.}
In this paper, we study the spectral gap property for transfer operators
of torus extensions over expanding maps.  
The systems $F: \TT^d\times \TT^\ell\to \TT^d\times \TT^\ell$ 
that we consider are of the form $F(x,y)=(Tx, y+\tau(x) \pmod{\ZZ^\ell})$, which 
are skew products with expanding $T: \TT^d\to \TT^d$ on the base 
and torus rotations on the fibers $\TT^\ell$
with rotation vectors $\tau(x)\in \RR^\ell$ for any $x\in \TT^d$.
We obtain the following dichotomy: either the transfer operator of $F$ 
has a spectral gap on a certain Hilbert space
 and therefore the system has exponential decay 
of correlations with respect to the smooth invariant measure, 
or the rotation function $\tau(x)$ over $\TT^d$ is an essential coboundary.
When the base map $T$ is fixed, the former case is open and dense in the $C^0$ topology of the space of skew products.
The latter implies that either the system is not weak mixing 
and is semiconjugate to an expanding 
endomorphism crossing a circle rotation, 
or it is unstably mixing and can be approximated by non-mixing skew products.

A Hilbert space, denoted by $\CW^{-s}$ for any $s<0$, is constructed using
the Sobolov spaces, which is contained in the distribution space and
contains all H\"older functions of H\"older exponents $|s|$ 
over the phase space $\TT^d\times \TT^\ell$ of $F$ (see Remark~\ref{Rspace}).  
It is well known that restricted to $L^2(\TT^{d+\ell})$,
the transfer operator  does not have spectral gap.
We define the Hilbert space $\CW^{-s}$ such that its norm is stronger along $\TT^d$-direction 
and weaker along $\TT^\ell$-direction.

The method we use to get the spectral gap and thus
the exponential mixing property is the semiclassical analysis. 
Instead of the Ruelle-Perron-Frobenius transfer operator, 
we use dual, the Koopman operator $\wF$, acting on 
the dual space $\CW^s$ of the Hilbert space $\CW^{-s}$.
By Fourier transform along $\TT^\ell$, the fiber direction, 
the operator $\wF$ can be decomposed to 
a family of operators $\{\wF_\bnu\}_{\bnu\in \ZZ^\ell}$, where $\bnu$ 
is the frequency.
Such operators can be regarded as Fourier integral operators.  
Using semiclassical analysis we show that if 
$\tau$ is not an essential coboundary, 
then the spectral radius of $\wF_\bnu$ is strictly less than $1$ for all $\bnu\not=\bzero$,
while $1$ is the only leading eigenvalue of $\wF_\bzero$ on the unit circle and it is simple. 
Moreover, Proposition~\ref{Pprop1} shows that the essential spectral radius of every $\wF_\bnu$
is uniformly bounded by a positive number $\rho_1\in (0, 1)$.
Further, we obtain in Proposition~\ref{Pprop2} that for all $\bnu$ of sufficiently large magnitude,
the operator norms of $\{\wF_\bnu^n\}_{n\in\NN}$ decay exponentially fast (in terms of $n$) 
with a uniform rate $\rho_2\in (0, 1)$.
We hence prove that the Koopman operator $\wF$ has a spectral gap on $\CW^{s}$, 
and so does the transfer operator on $\CW^{-s}$.  Therefore 
the system $F$ has exponential decay of correlations.

By analyzing the transfer operators, Dolgopyat established in \cite{MR1919377} 
the exponential mixing property for compact Lie group 
extensions of expanding maps under a generic condition called 
infinitesimally completely non-integrability, which
is equivalent to the non-coboundary condition of $\tau(x)$ if the group is a torus. 
The crucial technique used there is now called Dolgopyat's oscillatory cancellation argument,
and it has been successfully developed to study the rate of mixing for various systems with neutral direction, see \cite{MR1626749}, \cite{MR2113022},
\cite{MR2264836}, \cite{MR2964773}, etc.
Among all such results,  in \cite{Butterley-Eslami} Butterley and Eslami 
studied piecewise $C^2$ circle extensions $F$ and obtained the dichotomy
that either $F$ mixes exponentially, or $\tau$ is cohomologous to 
a piecewise constant.
We remark that their analysis did not provide a finer spectral structure of 
the transfer operators.\footnote{
\cite{Butterley-Eslami} only obtained non-uniform norm estimates for twisted transfer operators 
over the base map $T$, which is sufficient to establish exponential mixing of $F$ but not enough to prove
the spectral gap for the transfer operator associated to $F$.
}

In a somewhat different direction, the semiclassical analysis approach is used to study Ruelle-Pollicott resonances for some hyperbolic systems, see \cite{MR2285729}, \cite{MR2229997}, \cite{MR2461513}, \cite{MR3072166}, \cite{MR3729047}, \cite{MR3330427}, \cite{DyatGu16}, \cite{DyatZw16}, \cite{BonWeich18}, etc. 
Applying this approach in the context of partially hyperbolic systems,
Faure showed in \cite{MR2785978} that
a simple but intuitive model -- a circle 
extension of a circle expanding map -- has exponential decay of correlations under a so-called 
partially captive condition. 
It was recently pointed out in \cite{NTW15} that the partially captive condition 
is generic but much stronger than the non-coboundary condition of $\tau(x)$ for this two-dimensional model.
Using similar techniques, Arnoldi established 
in \cite{MR2924730} the asymptotic spectral gap and the fractal Weyl law for 
$\SU(2)$ extensions of circle expanding maps under the partially captive condition, and later Arnordi, Faure and 
Weich obtained a similar result in \cite{AFW13} for circle extensions of 
certain one-dimensional open expanding maps under a stronger condition called minimal captivity.
An improved estimate of the spectral gap was recently obtained by Faure and Weich 
\cite{MR3719542} for the model in \cite{AFW13}. 

Let us mention some similar results in the context of suspension semi-flows 
over expanding maps. Pollicott \cite{MR1685406}  
showed that a generic suspension semi-flow over an expanding Markov interval map
is exponentially mixing. In the case when the base is a linear expanding map,
Tsujii \cite{MR2380311} 
constructed an anisotropic Sobolev space
on which the transfer operator has spectral gap. 
Also, Baladi and Vall\'ee \cite{MR2113938} proved exponential mixing property for surface semi-flows without finite Markov partitions.
We also point out that the construction of anisotropic distributional space for 
hyperbolic diffeomorphisms, contact Anosov flows and other hyperbolic systems 
has been developed during the last two decades, see e.g. 
\cite{MR2313087}, \cite{MR2427585},  \cite{MR2652469},
\cite{MR2643889}, \cite{MR2964773}, \cite{MR3742756}, etc.

The main technique we use in this paper is the semiclassical 
analysis, inspired by Faure \cite{MR2785978} and other related works. 
A key ingredient in our analysis is that we introduce non-standard Sobolev spaces associated with dynamical weights. 
Although equivalent to the standard ones, these spaces are much more effective in extracting
the spectral properties of the Koopman operator $\hF$ and its decompositions $\{\hF_\bnu\}_{\bnu\in \ZZ^\ell}$. 
In fact, we convert each $\hF_\bnu^n$ by unitary conjugation into a new operator $\wQ_{\bnu, n}$
such that $\wQ_{\bnu, n}^*\wQ_{\bnu, n}$ is a pseudo-differential operator, 
whose symbol provides an upper bound for the operator norm of $\hF_\bnu^n$. 
Moreover, we prove directly that 
the rotation vector $\tau(x)$ is not an essential coboundary if and only if those upper bounds vanish uniformly exponentially fast as $n\to \infty$ for all high frequencies $\bnu$, from which we conclude that $\hF$ has spectral gap.
We remark that our approach bypasses the captive conditions, and has no dimension restrictions to either the base or the fiber.

This paper is organized as the following.  
The setting and statements of results are given in Section~1.
In Section~2 we introduce some notions and results from classical and
semiclassical analysis,
including Fourier transform, Sobolev spaces, Pseudo-differential operators,
Fourier Integral Operators, Egorov's Theorem, and $L^2$-continuity theorems.
This section is not necessary for the reader who is familar with the theory.
We prove the theorems of the paper in Section~3 based on
Proposition~\ref{Pprop1} and ~\ref{Pprop2}, which provides the spectral properties 
of the Koopman operator and its decompositions.
These two propositions are proved in Section~4, using the classical and semiclassical 
analysis.  A key estimates in the proof, stated in Lemma~\ref{Lmain estmt}, 
is postponed in Section~5.\\

\smallskip

Acknowledgement:  
We would like to thank Sheldon Newhouse, Zhenqi (Jenny) Wang and Zhengfang Zhou for their useful suggestions. 
We would also like to thank the anonymous referee for helpful comments.



\section{Statement of results.}
Let $\TT=\RR/\ZZ$, and
let $T: \TT^d\to \TT^d$ be a $C^\infty$ uniformly expanding map such that
\begin{equation}\label{fdef gamma}
\ga :=\inf_{(x, v)\in S\TT^d} |D_xT(v)|>1,
\end{equation}
where $S\TT^d$ is the unit tangent bundle over $\TT^d$. 
It is well known that $T$ has a unique smooth invariant probability measure 
$d\mu(x)=h(x)dx$, where the density function $h\in C^\infty(\TT^d, \RR^+)$.
Further, $T$ is mixing with respect to $\mu$. Here and throughout this paper, 
we fix the expanding map $T$.

Given a function $\tau\in C^\infty(\TT^d, \RR^\ell)$, we define 
the skew product $F: \TT^d\times \TT^\ell \to \TT^d\times \TT^\ell$ by
\begin{equation}\label{fdefF}
F \begin{pmatrix} x \\ y\end{pmatrix}= \begin{pmatrix} Tx 
    \ \ \ \ \ \ \ \ \ \ \ \ \ \ \ \   \\ 
     y+\tau(x)  \pmod{\ZZ^\ell}\end{pmatrix},
\end{equation}
which preserves the product measure $dA=d\mu(x) dy$. 
We also denote the skew product by $F_\tau$ when we aim to emphasize the 
rotation function $\tau$.

\begin{definition}
A real-valued function $\varphi\in C^\infty(\TT^d, \RR) $ 
is called an \emph{essential coboundary} over $T$ if there exist $c\in\RR$ and a measurable function 
$u: \TT^d\to \RR$ such that 
\begin{equation*}\label{fdefcobd}
\varphi(x)=c+u(x)-u(Tx), \ \ \mu-a.e. \ x.
\end{equation*}
Let $\fB$ be the space of real-valued essential coboundries over $T$.

A vector-valued function $\tau=(\tau_1, \tau_2, \dots, \tau_\ell)\in C^\infty(\TT^d, \RR^\ell) $ 
is called an \emph{essential coboundary} over $T$ if $\tau_1, \tau_2, \dots, \tau_\ell$ are linearly dependent $\mmod$ $\fB$, that is,
there exist 
$v\in\RR^\ell\backslash\{\bzero\}$, $c\in\RR$ and a measurable function 
$u: \TT^d\to \RR$ such that 
\begin{equation}\label{fdefcobd}
v\cdot \tau(x)=c+u(x)-u(Tx), \ \ \mu-a.e. \ x. \ \footnote{\ 
Here ``$\cdot$" denotes the standard inner product of two vectors in  $\RR^\ell$. In the rest of the paper, we shall abuse the notation $v\cdot w$ when one of $v$ and  $w$ belongs to $\ZZ^\ell$ or $\SSS^{\ell-1}$ - the unit sphere in $\RR^\ell$; that is, $v\cdot w$ represents the inner product of $v$ and $w$ as vectors in $\RR^\ell$. 
}
\end{equation}
\end{definition}

\begin{remark} $ $

{\rm (i)}
By Livsic theory (see e.g. \cite{MR1891682}), 
the measurable function $u:\TT^d\to \RR$ in \eqref{fdefcobd} is in fact of class $C^\infty$.

{\rm (ii)} 
The functions $\tau_1, \tau_2, \dots, \tau_\ell$ are called integrally dependent $\mmod$ $\fB$ if \eqref{fdefcobd} holds for some $v\in \ZZ^\ell\backslash\{\bzero\}$.
There are functions $\tau_1, \tau_2, \dots, \tau_\ell$ that are linearly dependent but not integrally dependent $\mmod$ $\fB$, unless $\ell=1$.
\end{remark}

Let $\CDD(\TT^{d+\ell})=C^\infty(\TT^{d+\ell})$.  
Its dual space $\CDD'(\TT^{d+\ell})$ is the space of distributions on 
$\TT^{d+\ell}$. 

The \emph{Koopman operator} 
$\wF: \CDD(\TT^{d+\ell})\to \CDD(\TT^{d+\ell})$ for $F$ is defined by 
$\wF\phi=\phi\circ F$.  
The \emph{(Ruelle-Perron-Frobenius) transfer operator} 
$\wF': \CDD'(\TT^{d+\ell})\to \CDD'(\TT^{d+\ell})$ is the dual operator of $\wF$,
defined by the duality 
\begin{equation}\label{fdualop}
(\wF' \psi)(\phi)=\psi(\wF\phi) \quad \text{for any}\  
\phi\in \CDD(\TT^{d+\ell}), \ \ \psi\in \CDD'(\TT^{d+\ell}).
\end{equation}

\begin{definition}\label{def spec gap}
We say that a linear operator $\CT: \CB\to \CB$ of a Banach space $\CB$
has a \emph{spectral gap} if its spectrum 
\begin{equation*}
\mathrm{Spec}(\CT)=\{1\}\cup \CK,
\end{equation*}
where $1$ is a simple eigenvalue and
$\CK$ is a compact subset of the unit open disk $\DD=\{z\in \CC: |z|<1\}$. 
\end{definition}


Our main result is the following.  

\begin{theoremmain}\label{ThmSpec gap transfer} 
Let $(\TT^{d+\ell}, F, dA)$ be the skew product as described above. 
Then the following dichotomy holds:
\begin{enumerate}
\item Either $\tau(x)$ is an essential coboundary over $T$; 
\item or for any $s <0$, there is an $\wF'$-invariant Hilbert space 
$\CW^{-s}$, which is contained in $\CDD'(\TT^{d+\ell})$ and contains $C^{-s}(\TT^{d+\ell})$, 
such that $\wF'|\CW^{-s}$ has a spectral gap.
\end{enumerate}
\end{theoremmain}

\begin{remark}
It is well known that $\wF'|L^2(\TT^{d+\ell})$ does not have spectral gap.
Statement (2) of Theorem~\ref{ThmSpec gap transfer} holds for 
some Hilbert space $\CW^{-s}$ whose norm 
is stronger along $\TT^d$-direction and weaker along 
$\TT^\ell$-direction than that of $L^2(\TT^{d+\ell})$ 
(see Definition~\ref{fdef space transfer} and Remark~\ref{Rspace}). 
\end{remark}

Since the spectrum of a linear operator on a Hilbert space
coincides with that of its dual operator on the dual space, 
we will instead prove the spectral gap for 
the Koopman operator in the latter case of the above dichotomy.
That is, 

\begin{theoremmain}\label{ThmSpec gap} 
The following dichotomy holds:
\begin{enumerate}
\item Either $\tau(x)$ is an essential coboundary over $T$; 
\item or for any $s<0$, there is an $\wF$-invariant Hilbert space 
$\CW^{s}$, which is contained in $\CDD'(\TT^{d+\ell})$ and contains $C^{-s}(\TT^{d+\ell})$, 
such that $\wF|\CW^{s}$ has a spectral gap.
\end{enumerate}
\end{theoremmain}

Theorem ~\ref{ThmSpec gap transfer} is equivalent to Theorem~\ref{ThmSpec gap}
if the space $\CW^{-s}$ is defined to be the dual space of $\CW^{s}$. 
We will specify the construction of the Hilbert space $\CW^{s}$ 
in Subsection \ref{SS CWs} (see \eqref{fdef space}), and prove 
the spectral gap for $\wF|\CW^{s}$ in Section \ref{SSproof spec gap}.

\begin{remark}\label{rem generic}
The first case in Theorems~\ref{ThmSpec gap transfer} and~\ref{ThmSpec gap} is very rare in the sense that the closed subspace that consists of all essential coboundaries has infinite codimension in $C^\infty(\TT^d, \RR^\ell)$. It means that 
the second case is generic in the space of skew products, i.e., there is a
subset $\CU$ in $C^\infty(\TT^d, \RR^\ell)$ that is
open and dense in the $C^0$ topology such that 
for all $\tau\in \CU$, the transfer operator of 
the corresponding skew product $F_\tau$ given in \eqref{fdefF} has a spectral gap.

The infinite codimension of essential coboundaries is a crucial property in showing the stable ergodicity of skew products over general hyperbolic systems. Among tremendous results on this topic, we refer the reader to \cite{MR1632190}, \cite{MR1644099},\cite{MR1717580}, \cite{MR2129109}, etc.
\end{remark}

The mixing property of the system $(\TT^{d+\ell}, F, dA)$ is quantified 
by the rates of decay  of correlations. 
We say that the skew product $F$ is \emph{exponentially mixing} with respect 
to the smooth measure $dA$ 
if there exists $\rho\in [0, 1)$ such that 
for any pair of H\"older observables 
$\phi, \psi\in C^\alpha(\TT^{d+\ell})$, $\alpha>0$, 
the correlation function 
\begin{equation*}
C_n(\phi, \psi; F, dA)
=\left|\int \phi\circ F^n \cdot \psi dA- \int \phi dA \int \psi dA\right|
\end{equation*}
satisfies $C_n(\phi, \psi; F, dA)\le C_{\phi, \psi} \rho^n$
for all $n\ge 1$,  
where $C_{\phi, \psi}>0$ is a constant depending on $\phi$ and $\psi$.

\begin{theoremmain}\label{ThmMain thm} 
Let $F=F_\tau: \TT^d\times \TT^\ell \to \TT^d\times \TT^\ell$ be defined as 
in \eqref{fdefF}.  
If $\tau(x)$ is not an essential coboundary over $T$, 
then $F$ is exponentially mixing with respect to $dA$ 
for any pair of H\"older observables $\phi, \psi\in C^{-s}(\TT^{d+\ell})$ 
for any $-s>0$. 
\end{theoremmain}

\begin{remark} $ $
Theorem~\ref{ThmMain thm} follows directly from Theorem~\ref{ThmSpec gap}.
The conclusions of Theorem~\ref{ThmSpec gap} and Theorem~\ref{ThmMain thm} 
immediately imply the following dichotomy:
\begin{enumerate}
\item Either $F$ is exponentially mixing with respect to $dA$;
\item Or $\tau(x)$ is an essential coboundary over $T$.
\end{enumerate}
\end{remark}

If $d=\ell =1$, the dichotomy in the remark is proved by Butterley and Eslami 
\cite{Butterley-Eslami}, in which the circle expansion $T$ 
and the rotation $\tau$ are allowed to have a finite number of discontinuities.

In our context, we say that $F=F_\tau$ is stably ergodic if $F_{\tau'}$ is ergodic for any $\tau'$ that is $C^0$-close to $\tau$. The stable mixing property and stable exponential mixing property are defined in a similar fashion.
 
It was shown by Parry and Pollicott \cite{MR1632190}, and also by Field and Parry \cite{MR1644099}, that 
$(\TT^{d+\ell}, F, dA)$ is weak mixing (or stably mixing) 
if and only if the functions $\tau_1, \tau_2, \dots, \tau_\ell$ are integrally independent 
(or linearly independent) $\mmod$ $\fB$.\footnote{\ 
Originally in \cite{MR1644099}, the independence is modulo $\mathbf{V}+\fB$ for some finite-dimensional subspace $\mathbf{V}$ of $C^\infty(\TT^d, \RR)$. In our setting, $\mathbf{V}=\{0\}$ because the rotation function $\tau$ is null-homotopic in $C^\infty(\TT^d, \RR^\ell)$.
}
In other words, Theorem~\ref{ThmMain thm} asserts that 
if $F$ is stably mixing, then it is exponentially mixing, and furthermore, it is stably exponentially mixing by Remark~\ref{rem generic}.
We can say more about the ergodic properties of the skew product  over an expanding map:  
Dolgopyat \cite{MR1919377} proved that $F$ is stably ergodic if and only if it is exponentially mixing;
Field and Parry \cite{MR1644099} showed that stable ergodicity implies stable mixing property for skew products.
Combining all these results and Theorem~\ref{ThmMain thm}, we immediately obtain the following corollary.
 
\begin{cor} Let $F$ be the skew product given by \eqref{fdefF}. The following statements are equivalent:
\begin{enumerate}
\item $F$ is stably ergodic;
\item $F$ is stably mixing;
\item $F$ is exponentially mixing;
\item $F$ is stably exponentially mixing.
\end{enumerate}
\end{cor} 
 
\begin{remark} $ $

{\rm (i)} In the case when $\ell=1$, if $F$ is mixing, then $F$ is stably mixing and thus (stably) exponentially mixing. This is simply because
that integral independence and linear independence $\mmod$ $\fB$ are the same for $\tau\in C^\infty(\TT^d, \RR)$.

{\rm (ii)}
Let $F_0: \TT^3\to \TT^3$ be given by
\begin{equation*}
F_0(x, y_1, y_2)=(2x, y_1+\tau_0(x), y_2+\sqrt3 \tau_0(x)) \pmod{\ZZ^3},
\end{equation*}
where $\tau_0(x)$ is not a real-valued essential coboundary over the linear expanding map $x\mapsto 2x \pmod \ZZ$ of the circle. It is clear that  $\tau_0(x)$ and $\sqrt{3}\tau_0(x)$ are integrally independent but not linearly independent $\mmod$ $\fB$, and hence $F_0$ is unstably mixing and thus unstably ergodic by \cite{MR1632190}, \cite{MR1644099}.

{\rm (iii)} In \cite{MR3772032}, Zhang  considered a circle extension $F$ of a linear circle endomorphism,
and showed that if $F$ is stably ergodic in $C(\TT^2, \TT^2)$, then the SRB densities vary smoothly for maps 
in a neighborhood of $F$.
The technique therein suggests that we may be able to obtain the dichotomy in Theorem~\ref{ThmSpec gap transfer}
for small perturbations of $F$, which need not be skew products.
\end{remark}

\smallskip

Next we characterize the dynamical properties of $F=F_\tau$ when the rotation vector $\tau=(\tau_1, \tau_2, \dots, \tau_\ell)$ 
is an essential coboundary, that is, the functions $\tau_1, \tau_2, \dots, \tau_\ell$ are linearly dependent $\mmod$ $\fB$. There are two cases:
\begin{enumerate}
\item If $\tau_1, \tau_2, \dots, \tau_\ell$ are integrally dependent $\mmod$ $\fB$, then the behaviors of $F_\tau$ in the 
$\TT^\ell$ direction become very simple, as we see in Part (iii) of the next theorem. In particular, $F_\tau$ is not weak mixing.
\item If $\tau_1, \tau_2, \dots, \tau_\ell$ are linearly dependent but integrally independent $\mmod$ $\fB$, then $F_\tau$ is unstably mixing. 
We can approximate $F_\tau$ by a sequence of non-mixing skew products $F_{\tau(n)}$ as follows. Pick real-valued sequences $\{c_{n, i}\}_{n\in \NN}$, $i=1, 2, \dots, \ell$, such that $\lim_{n\to\infty} c_{n,i}=1$ and $c_{n, 1}\tau_1, c_{n, 2}\tau_2, \dots, c_{n,\ell} \tau_\ell$ are integrally dependent $\mmod$ $\fB$.
Then set $\tau(n)=(c_{n, 1}\tau_1, c_{n, 2}\tau_2, \dots, c_{n,\ell} \tau_\ell)$.
\end{enumerate}

A foliation $\CL$ of a smooth manifold $M$ is of dimension $m$ 
if the leaves of $\CL$ are $m$ dimensional submanifolds.
For a smooth dynamical system $(F, M)$, a foliation $\CL$ of $M$ is 
$F$ invariant if $F$ preserves the leaves, that is, $F(\CL(z))=\CL(F(z))$ 
for any $z\in M$, where $\CL(z)$ is the leaf of $\CL$ containing $z$.  



A smooth dynamical system $(F, M)$ is semiconjugate to 
a smooth system $(G, N)$ if there is a smooth map $\pi: M\to N$ such that 
$\pi\circ F=G\circ \pi$.

\begin{theoremmain}\label{ThmCobd}
Let $F=F_\tau: \TT^d\times \TT^\ell \to \TT^d\times \TT^\ell$ be defined as 
in \eqref{fdefF}.  The following conditions are equivalent.
\begin{enumerate}[{\rm (i)}]
\item
$\tau_1, \tau_2, \dots, \tau_\ell$ are integrally dependent $\mmod$ $\fB$;

\item  There is an $F$ invariant $d+\ell -1$ dimensional 
foliation $\CL$ of $\TT^d\times \TT^\ell$ and a vector 
$v\in \ZZ^\ell\setminus \{0\}$ such that restricted to each fiber
$\{x\}\times \TT^\ell$, the leaves of $\CL|_{\{x\}\times \TT^\ell}$ 
are $\ell -1$ dimensional and normal to $v$.

\item $F$ is semiconjugate to the map 
$G=T\times R_{c}: \TT^d\times \TT\to \TT^d\times \TT$ 
through a continuous map $\pi: \TT^d\times \TT^\ell \to \TT^d\times \TT$, 
where $R_{c}:\TT\to \TT$
is a circle rotation with rotation number $c\in \RR$.
Further, $F$ is semiconjugate to $R_{c}$.

\item $F$ is not weak mixing.
\end{enumerate}
\end{theoremmain}



\section{Semiclassical Analysis: Preliminaries}

In this section we introduce some notions and basic properties
in semiclassical analysis which we are going to use.   
The distribution spaces and Sobolev spaces will be used in construction 
of the Hilbert space $\CW^{s}$ in Theorem \ref{ThmSpec gap}. 
The pseudo-differential operators (PDO) and Fourier integral operators (FIO)
will be used to prove Proposition~\ref{Pprop1} and \ref{Pprop2},
where the Egorov's theorems and the $L^2$-continuity theorems are also used.
For more information and details on the general theory of PDOs and FIOs, one can see in standard
references (e.g. \cite{MR2744150, MR1872698, MR872855, MR2952218}).

\subsection{Function spaces}

In this subsection we introduce the notion of distribution spaces and 
Sobolev spaces, and then construct the Hilbert space $\CW^{s}$ in 
Theorem~\ref{ThmSpec gap} using Sobolev spaces.
\subsubsection{Distribution spaces and Sobolev spaces}
Let $\CDD(\TT^{d+\ell})=C^\infty(\TT^{d+\ell})$.  
Its dual space $\CDD'(\TT^{d+\ell})$ is the space of distributions 
on $\TT^{d+\ell}$. 
There is a natural injection 
$i:\CDD(\TT^{d+\ell})\hookrightarrow \CDD'(\TT^{d+\ell})$ given by
\begin{equation}\label{fdualfun}
\psi(\varphi):=\int_{\TT^{d+\ell}} \psi(x,y)\varphi(x,y) \ dxdy 
\end{equation}
for any $\varphi, \psi\in \CDD(\TT^{d+\ell})$.
Hence we can regard that $\CDD(\TT^{d+\ell})\subset \CDD'(\TT^{d+\ell})$.\\

The \emph{Fourier transform} of $\varphi\in L^2(\TT^d)$ is defined by 
\begin{equation}\label{fdefFTransf}
\wht{\varphi}(\xi)=\int_{\TT^d} \varphi(x) e^{-i2\pi x\cdot \xi } dx, 
\quad \ \xi\in \ZZ^d. 
\end{equation}
Note that the Fourier transform is an isometry from $L^2(\TT^d)$ to $\ell^2(\ZZ^d)$,
and the \emph{inverse transform} is given by
\begin{equation}\label{fdefinvFTransf}
{\varphi}(x)=\sum_{\xi\in \ZZ^d}\wht\varphi(\xi) e^{i2\pi \xi\cdot x}, 
\quad \ x\in \TT^d. 
\end{equation}
It is well known that if $\varphi\in \CDD(\TT^d)$, then $\wht{\varphi}(\xi)$ converges to zero super-polynomially fast
as $|\xi|\to \infty$. 

Let $\om$ be the counting measure over the lattice $\ZZ^d$ on $\RR^d$,   
i.e., $\disp \om(\xi)=\sum_{n\in \ZZ^d} \de(\xi-n)$ for $\xi\in \RR^d$. 
Then by the above equations we have
\begin{equation}\label{fFTransfID}
{\varphi}(x)=\int_{\RR^d} \wht\varphi(\xi) e^{i2\pi \xi\cdot x}d\om(\xi) 
=\int_{\RR^d} \int_{\TT^d}  e^{i2\pi(x-y)\cdot  \xi}\varphi(y) dy d\om(\xi),
\ \ x\in \TT^d. 
\end{equation}


Denote $\lan \xi\ran=\sqrt{1+|\xi|^2}$, and introduce the standard $s$-inner product 
\begin{equation}\label{fdefinnerprod}
\lan \varphi, \psi\ran_{s}
=\sum_{\xi\in \ZZ^d} 
\lan\xi\ran^{2s} \wht{\varphi}(\xi)\overline{\wht{\psi}(\xi)}, 
\quad \ \varphi, \psi\in \CDD(\TT^d)
\end{equation}
for any $s\in \RR$.
The Sobolev space $H^s(\TT^d)$ is the completion of $\CDD(\TT^d)$ 
under $\lan\cdot, \cdot\ran_s$. 

\begin{prop}\label{PSob space}
Sobolev spaces have the following properties:
\begin{itemize}
\item[(i)] 
$\CDD(\TT^d)\subset H^s(\TT^d) \subset \CDD'(\TT^d)$ for any $s\in \RR$;

\item[(ii)] 
$H^0(\TT^d)=L^2(\TT^d)$, and $H^s(\TT^d)
=\{\varphi: D_x^\beta \varphi\in L^2(\TT^d)\ \text{for any}\ |\beta|\le s\}$
if $s\in \NN$, where $D_x^\beta\varphi$ are weak derivatives of $\varphi$;

\item[(iii)] 
$H^s(\TT^d)\subset H^{s'}(\TT^d)$ if $s>s'$;

\item[(iv)] 
$C^s(\TT^d)\subset H^s(\TT^d)$ for all $s\ge 0$, and if $s>\frac{d}{2}$, 
then $H^s(\TT^d)\subset C^{s-\frac{d}{2}-\ve}(\TT^d)$ for any small $\ve>0$;

\item[(v)] 
the dual space of $H^s(\TT^d)$, $s>0$, is $H^{-s}(\TT^d)$, 
and the dual action of $\phi\in L^2(\TT^d)\subset H^{-s}(\TT^d)$ 
on the function $\psi\in H^s(\TT^d)$ is given by (\ref{fdualfun}).
\end{itemize}
\end{prop}

For technical treatments, besides the standard $s$-inner product given in \eqref{fdefinnerprod}, we will also use $t$-scaled $s$-inner product on $H^s(\TT^d)$ for $t>0$, that is, 
\begin{equation}\label{fdef innerprod t scale}
\lan\varphi, \psi\ran_{s, t}=t^{2s}\lan\varphi, \psi\ran_s=\sum_{\xi\in \ZZ^d} t^{2s}
\lan\xi\ran^{2s} \wht{\varphi}(\xi)\overline{\wht{\psi}(\xi)}.
\end{equation}
When equipped with $\lan\cdot, \cdot\ran_{s,t}$, the space is denoted by $H^s_t(\TT^d)$.
See Section~\ref{Sec decomp Koopman} for our particular choices of the scaling factor $t$.
We shall also introduce another different but equivalent inner product on $H^s(\TT^d)$ in Section~\ref{SSsp op}.

\subsubsection{The Hilbert space $\CW^{s}$ and $\CW^{-s}$}\label{SS CWs}
The Hilbert space $\CW^{s}$ which we will use in Theorem~\ref{ThmSpec gap} 
is of the form 
\begin{equation}\label{fdef space}
\CW^{s}=H^{s}(\TT^d)\otimes H^{-s}(\TT^\ell),
\end{equation}
where $s<0$, equipped with the inner product given by 
\begin{equation*}
\lan \varphi_1\otimes \psi_1, \varphi_2\otimes \psi_2\ran_{\CW^{s}}
=\lan \varphi_1, \varphi_2\ran_{H^{s}(\TT^d)}\ \lan \psi_1, \psi_2\ran_{H^{-s}(\TT^\ell)}
\end{equation*}
and extended by linearity.\footnote{\
In this paper, the tensor product of two Banach or Hilbert spaces always refers to the metric space completion of their algebraic tensor product.} 
We shall give a more explicit formula of $\lan \cdot, \cdot\ran_{\CW^{s}}$ in Section \ref{Sec decomp Koopman}.

Using the duality \eqref{fdualfun}, we see that the dual Hilbert space is
\begin{equation}\label{fdef space transfer}
\CW^{-s}=(\CW^s)'=H^{-s}(\TT^d)\otimes H^{s}(\TT^\ell).
\end{equation}

\begin{remark}\label{Rspace}
By Proposition~\ref{PSob space}(ii) and (iii), we have 
$L^2(\TT^d)\subset H^s(\TT^d)$ and $C^{-s}(\TT^\ell)\subset H^{-s}(\TT^\ell)$ when $s<0$, thus 
$$\CW^s\supset  
L^2(\TT^d)\otimes C^{-s}(\TT^\ell)\supset  C^{-s}(\TT^{d+\ell}).$$
Similarly, the space $\CW^{-s}$ contains $C^{-s}(\TT^{d+\ell})$ as well. 
By Proposition~\ref{PSob space}(i), it is obvious that $\CW^s$ and $\CW^{-s}$ 
are both contained in $\CDD'(\TT^{d+\ell})$.
\end{remark}

Recall that the Koopman operator 
$\wF: \CDD(\TT^{d+\ell})\to \CDD(\TT^{d+\ell})$ for $F$ is defined by 
$\wF\phi=\phi\circ F$, and the dual operator 
$\wF': \CDD'(\TT^{d+\ell})\to \CDD'(\TT^{d+\ell})$ 
is the RPF (Ruelle-Perron-Frobenius) transfer operator
given by \eqref{fdualop}.  
With the duality given by \eqref{fdualfun}, 
the transfer operator can be explicitly expressed as
\begin{equation*}\label{ftransfop}
\wF'\psi(x, y)
=\sum_{F(z, w)=(x, y)} \dfrac{\psi(z, w)}{|\Jac(F)(z, w)|}
\quad \text{for any}\ \psi\in \CDD(\TT^{d+\ell}).
\end{equation*}
See \cite{MR2129258} for more details.

By the duality \eqref{fdualfun} again, we can extend the domain of 
the Koopman operator to $\CDD'(\TT^{d+\ell})$ such that for any 
$\psi\in \CDD'(\TT^{d+\ell})$, $\wF \psi$ is the distribution satisfying 
\begin{equation*}
(\wF \psi)(\phi)=\psi(\wF' \phi) \quad \text{for any}\  
\phi\in \CDD(\TT^{d+\ell}).
\end{equation*}
It is easy to see that $\wF$ can act on $\CW^{s}$
and $\wF'$ can act on $\CW^{-s}$.





\subsection{Semiclassical analysis on the torus}

In this subsection we introduce the pseudo-differential operators 
and Fourier integral operators.
The underlying manifold that we analyze is the torus. 
In standard references the quantization on manifolds is usually defined
locally in charts.   
For the global definitions on tori we recommend  Chapter 4 
in \cite{MR2567604}.

\subsubsection{Symbols}

The cotangent bundle over $\TT^d$ can be identified as 
$T^*\TT^d\cong \TT^d\times \RR^d$. 
Denote $\NN_0=\NN\cup\{0\}$ and $\NN_0^d=(\NN_0)^d$.

\begin{definition}
A complex-valued function $a\in C^\infty(T^*\TT^d)$ is called 
a \emph{classical symbol of order $m\in \RR$} on $T^*\TT^d$ if 
\begin{equation}\label{def seminorm}
\CN_{\al\beta, m}(a):= \sup_{(x,\xi)\in T^*\TT^d} 
\dfrac{|\p_x^\al\p_\xi^\beta a(x,\xi)|}{ \lan\xi\ran^{m-|\beta|} }<\infty
\end{equation}
for any $\al, \beta\in \NN_0^d$,   where $\lan\xi\ran=\sqrt{1+|\xi|^2}$. 

A complex-valued function $a\in C^\infty(T^*\TT^d \times (0,1] )$ is called 
a \emph{semiclassical symbol of order $m\in \RR$} on $T^*\TT^d$  
if \eqref{def seminorm} holds uniformly in $\hb$ with $a(x,\xi)$
replaced by $a(x,\xi; \hb)$.\footnote{
In the general theory of semiclassical analysis, $\hb\ll 1$ is the 
\emph{Planck's constant} parametrizing the whole family of symbol functions 
and thus the symbol calculus for corresponding semiclassical 
pseudo-differential operators. }

The space of symbols of order $m$, either classical or semiclassical, 
is denoted as $S^m$, which is short for $S^m(T^*\TT^d)$.
\end{definition}

The topology on the space $S^m$ is generated by the seminorms 
$\{\CN_{\al\beta, m}(\cdot)\}_{\al, \beta\in \NN_0^d}$ given by \eqref{def seminorm}.  
For any $k\in \NN_0$, 
we denote $\CN_{k, m}(a)=\sup_{|\al|+|\beta|\le k} \CN_{\al\beta, m}(a).$
We often write $\CN_k(a)$ for short if the order of $a$ is clear. 

Note that if $a\in S^m$ and $b\in S^{m'}$, then $a+b\in S^{\max\{m, m'\}}$, 
and $ab\in S^{m+m'}$.
Also, for any $a\in S^m$ and $\al\in \NN_0^d$, $\p_x^\al a\in S^m$ and 
$\p_\xi^\al a\in S^{m-|\al|}$.
For these operations, we have corresponding seminorm relations, 
for instance, $\CN_{k}(\p_x^\al a)\le \CN_{k+|\al|}(a)$.

\begin{remark}
The standard toroidal symbols are defined on $\TT^d\times \ZZ^d$.  
Here we use symbols $a(x, \xi)$ defined on $T^*\TT^d\cong \TT^d\times \RR^d$ 
since we need to use $a(x, \hb\xi)$ for $\hb\in (0,1]$ as well
(see Definition~\ref{DefPDO} and \ref{DefhFIO}).
However, the two different domains give the same family of symbols, 
because a symbol $\tilde a\in S^m(\TT^d\times \ZZ^d)$ 
is a toroidal symbol 
if and only if there exists a symbol $a\in S^m(\TT^d\times \RR^d)$ 
such that $\tilde a = a|_{\TT^d\times \ZZ^d}$  (see e.g. \cite{MR2567604}, Theorem~4.5.3).
\end{remark}

\subsubsection{Pseudo-differential operators}

Let $\hb\in (0,1]$.

\begin{definition}\label{DefPDO} 
Given a symbol $a\in S^m$, the linear operator $\Op_\hb(a): \CDD(\TT^d)\to \CDD(\TT^d)$ defined by
\begin{equation}\label{fdefhPDO}
\begin{split}
\Op_\hb(a)\varphi(x)
=& \int_{T^*\TT^d} a(x,\xi) e^{i2\pi\frac{\xi}{\hb}\cdot(x-y)} \varphi(y) dyd \left(\frac{\xi}{\hb}\right) \\ 
=&\int_{T^*\TT^d} a(x,\hb\xi) e^{i2\pi\xi\cdot(x-y)} \varphi(y) dyd\xi 
\end{split}
\end{equation}
is called a (toroidal) $\hb$-scaled \emph{pseudo-differential operator (PDO)} 
of order $m$ corresponding to the symbol $a\in S^m$. 
We denote the space of $\hb$-scaled PDOs of order $m$ by $\OP_\hb S^m$. 
\end{definition}

It is easy to check that \eqref{fdefhPDO} implies that $\Op_\hb(a)\left(C^\infty_c(\TT^d) \right)\subset \CDD(\TT^d)$,
and thus $\Op_\hb(a)\left(\CDD(\TT^d) \right)\subset \CDD(\TT^d)$ by the compactness of $\TT^d$.
Hence $\Op_\hb(a)$ is a well-defined operator from $\CDD(\TT^d)$ to itself.

By standard duality argument, we extend 
$\Op_\hb (a): \CDD'(\TT^d)\to \CDD'(\TT^d)$. 
That is, if we let $\Op_\hb(a)': \CDD'(\TT^d)\to \CDD'(\TT^d)$ be 
the dual operator of $\Op_\hb(a)$, then for any $\varphi\in \CDD'(\TT^d)$,
$\Op_\hb (a)'\varphi$ is defined by the formula
\begin{equation*}
(\Op_\hb (a)'\varphi)(\psi)
=\varphi(\Op_\hb (a)\psi)   \qquad \forall \psi\in \CDD(\TT^d).
\end{equation*}
Note that $\varphi(\psi)=\int_{\TT^d} \varphi(x) \psi(x) dx$ if $\varphi\in \CDD(\TT^d)$ 
is regarded as an element in $\CDD'(\TT^d)$.  It is then not hard to obtain 
\begin{equation*}
\Op_\hb(a)'\varphi(x)=\int_{T^*\TT^d} a(y, \hb\xi) e^{i2\pi\xi\cdot(x-y)} \varphi(y) dyd\xi 
\qquad \forall \varphi\in \CDD(\TT^d),
\end{equation*}
which implies that the dual operator $\Op_\hb (a)'$ preserves the space $\CDD(\TT^d)$. 
Therefore, by the duality, for any  $\psi\in \CDD'(\TT^d)$, 
$\Op_\hb (a)\psi \in \CDD'(\TT^d)$ can be defined by
\begin{equation*}
(\Op_\hb (a)\psi)(\varphi)=\psi(\Op_\hb (a)'\varphi) \qquad 
\forall \varphi\in \CDD(\TT^d).
\end{equation*}

Moreover, for any $s\in \RR$ and any symbol $a\in S^m$,
the $\hb$-scaled PDO $\Op_\hb(a): H^s(\TT^d)\to H^{s-m}(\TT^d)$ is a bounded operator. 
This fact is proven for $\RR^d$ in  Proposition 5.5 of \cite{MR2743652},
for which the proof can be easily modified to $\TT^d$.


The formula with $\hb=1$ in (\ref{fdefhPDO}) gives 
the definition of classical pseudo-differential operator $\Op(a)=\Op_1(a)$. We denote $\OP S^m=\OP_1 S^m$. 
In this way, the $\hb$-scaled PDO with semiclassical symbol $a\in S^m$ 
can be regarded as the classical PDO with symbol $a_\hb\in S^m$, 
that is, $\Op_\hb(a)=\Op(a_\hb)$, where $a_\hb(x,\xi)=a(x, \hb \xi; \hb)$. 

We see by \eqref{fFTransfID},
that if $a(x,\xi)=1$, then $\Op(a)=\Id$; and if $a(x,\xi)=i2\pi\xi_k$ for some $k=1, 2, \dots, d$, 
then $\disp \Op(a)=\frac{\partial}{\partial x_k}$.

\subsubsection{Fourier integral operators}

\begin{definition}\label{DefhFIO}
A (toroidal) $\hb$-scaled \emph{Fourier integral operator (FIO)}
$\Phi_\hb:\CDD(\TT^d)\to \CDD(\TT^d)$ with \emph{amplitude function $a\in S^m $}
and a real-valued \emph{phase function $S\in S^1$}
is of the form
\begin{equation*}\label{fdefhFIO}
\begin{split}
\Phi_\hb\varphi(x)=\Phi_\hb(a, S) \varphi(x)
=& \int_{T^*\TT^d} a(x, \xi) e^{i 2\pi\frac{1}{\hb}[S(x,\xi)-y\cdot \xi]} \varphi(y) dyd\left(\frac{\xi}{\hb}\right) \\
=& \int_{T^*\TT^d} a(x, \hb\xi) e^{i 2\pi[\frac{1}{\hb}S(x, \hb\xi)-y\cdot \xi]} \varphi(y) dyd\xi,
\end{split}
\end{equation*}
where the \emph{phase function} $S(x,\xi)$ satisfies the following conditions:
\begin{enumerate}
\item\label{cond regularity} there are $c_1, c_2>0$ such that $\left|\dfrac{\p S(x,\xi)}{ \p x}\right|\ge c_1|\xi|$ for all $(x,\xi)$ with $|\xi|\ge c_2$;
\item $S(x,\xi)$ is strongly non-degenerate, i.e., there is $c_3>0$ such that 
\begin{equation*}
\left|\det\left(\dfrac{\p^2 S(x,\xi)}{\p x\p \xi}\right)\right|\ge c_3 \quad \text{for any}\ (x,\xi)\in T^*\RR^d.
\end{equation*}
\end{enumerate}
\end{definition}

Note that the \emph{classical Fourier integral operator} $\Phi=\Phi_1$ 
is the one with $\hb=1$. 

By standard duality argument, we can extend $\Phi_\hb: \CDD'(\TT^d)\to \CDD'(\TT^d)$. 
Further, $\Phi_\hb: H^s(\TT^d)\to H^{s-m}(\TT^d)$ is a bounded operator if its amplitude $a\in S^m$. 
This fact is proven for $\RR^d$ in Proposition 3.1 of \cite{MR2243967},
for which the proof can be easily modified to $\TT^d$.

\begin{remark}

{\rm (i)} H\"ormander's definition of phase functions usually assumes the homogeneity of degree one in $\xi$. Following Egorov \cite{MR872855}, we replace the homogeneity by that $S\in S^1$ and Condition (\ref{cond regularity}).

{\rm (ii)} If we take $S(x,\xi)=x\cdot \xi$, then $\Phi_\hb(a, S)$ becomes 
an $\hb$-scaled pseudo-differential operator with the symbol $a$.
\end{remark}

\begin{definition} The canonical transformation associated to an $\hb$-scaled FIO 
with phase $S$ is the transformation $\CF_\hb: T^*\TT^d \to T^*\TT^d$
which sends $(x, \hb\xi)$ to $(y, \hb\eta)$, such that 
\begin{equation}\label{def canonical FIO}
y=\dfrac{\p S(x,\eta)}{\p \eta}, \qquad\ \xi=\dfrac{\p S(x,\eta)}{ \p x}. 
\end{equation}
In other words, the phase function $S$ serves as the generating function of the canonical transformation.
\end{definition}
In the classical case when $\hb=1$, we write $\CF=\CF_1$.

\subsection{Some Theorems}
Now we introduce some important theorems in semiclassical analysis
and write them in a form that we will use to prove 
Proposition~\ref{Pprop1} and \ref{Pprop2}.

\subsubsection{The symbol calculus}

If $m<m'$, then $S^m\subset S^{m'}$ and $\OP_\hb S^m\subset \OP_\hb S^{m'}$. Set $S^{-\infty}=\bigcap_{m\in \RR} S^m$. If $a\in S^{-\infty}$, then $\Op_\hb (a)$ is a smoothing (and hence compact) operator. 

Given two classical symbols $a, b\in S^m$, if the difference $a-b\in S^{m'}$ for some $-\infty\le m'<m$, we write $a=b \pmod{S^{m'}}$. 

Given two semiclassical symbols $a=a(x, \xi; \hb)$ and $b=b(x, \xi; \hb)$ in $S^m$, if 
\begin{equation*}
a(x, \xi; \hb)-b(x, \xi; \hb)=\hb^{m-m'} r(x, \xi; \hb)
\end{equation*}
for some $r\in S^{m'}$, we shall denote as $a=b \pmod {\hb^{m-m'} S^{m'}}$ for short. 

For classical PDOs, the following results are standard.

\begin{theorem}\label{Thm symbol cal} 
\begin{enumerate}
\item
Adjoint: If $A\in\OP S^m$ has a symbol $a$, then the adjoint operator $A^*\in \OP S^m$ has a symbol $a^*=\oa \pmod{S^{m-1}}$.
\item Composition: If $A\in\OP S^m$ has a symbol $a$ and $B\in \OP S^{m'}$ has a symbol $b$, then the compositions $A\circ B\in \OP S^{m+m'}$ has a symbol $a\#b=ab \pmod{ S^{m+m'-1}}$.
\item Inverse: If $A\in \OP S^m$ has an elliptic\footnote{
A symbol $a\in S^m$ is \emph{elliptic} if there are $c>0$ and $R>0$ such that 
$|a(x, \xi)|\ge c\lan \xi\ran^m$ for all $(x, \xi)$ with $|\xi|>R$.
} symbol $a$ and is invertible, then $A^{-1}\in\OP S^{-m}$ has a symbol $a^{-1} \pmod{ S^{-m-1}}$. 
\end{enumerate}
Moreover, the $\CN_k$-seminorm of all the remainders in the above modulo class only depends on the $\CN_{k+2}$-seminorm of the original symbols.
\end{theorem}

Recall that $\OP_\hb(a)$ can be regarded as $\OP(a_\hb)$, where $a_\hb(x,\xi)=a(x, \hb\xi)$. As a direct consequence of the above Theorem and standard symbol calculus,
we have the following rules of the symbol calculus for the $\hb$-scaled PDOs.

\begin{theorem}\label{Thm symbol cal hb} 
\begin{enumerate}
\item Adjoint: If $A\in\OP_\hb S^m$ has a symbol $a$, then the adjoint operator $A^*\in \OP_\hb S^m$ has a symbol 
$a^*=\oa \pmod{\hb S^{m-1}}.$
\item Composition: If $A\in\OP_\hb S^m$ has a symbol $a$ and $B\in \OP_\hb S^{m'}$ has a symbol $b$, then the compositions $A\circ B\in \OP_\hb S^{m+m'}$ has a symbol 
$a\#b= ab\pmod{\hb S^{m+m'-1}}.$
\item Inverse: If $A\in \OP_\hb S^m$ has an elliptic symbol $a$ and is invertible, then $A^{-1}\in \OP_\hb S^{-m}$ has a symbol $a^{-1} \pmod{ \hb S^{-m-1}}$.
\end{enumerate}
Moreover, if $\hb r$ is one of the remainders in the above modulo class, then the seminorm
$\CN_k(r)$ only depends on $\CN_{k+2}(a)$.
\end{theorem}

\subsubsection{Egorov's Theorem}

Let $\Om$ be an open domain in $\TT^d$. We say that a symbol $a\in S^m$ is supported in $\Om\times \RR^d$ if $a(x,\xi)=0$ for any $(x,\xi)\in (\TT^d\backslash \Om)\times \RR^d$. 
The class of such symbols is denoted by $S^m(\Om\times \RR^d)$.

We first state the original version of classical Egorov's theorem in \cite{MR872855} for the invertible case.

\begin{theorem}\label{ThmEgorov1} Let $A\in \OP S^m$ with symbol $a\in S^m(\Om\times \RR^d)$, and $\Phi$ be a classical FIO with amplitude $b\in S^0$ and phase $S$. Let $\CF(x,\xi)=(y,\eta)$ be the canonical transformation associated to $\Phi$ (as defined in (\ref{def canonical FIO}) with $\hb=1$), and $\Om'$ be the image of $\Om\times \RR^d$ under the first $d$ components of $\CF$. We assume that 
$\CF: \Om\times \RR^d\to \Om'\times \RR^d$ is a bijective map. Then the operator $\Phi^* A \Phi\in \OP S^m$ has a symbol $\wa\in S^m(\Om'\times \RR^d)$ such that
$$\wa(y,\eta)=\wa(\CF(x,\xi))= a(x,\xi) |b(x,\xi)|^2 \left|\det \left(\dfrac{\p^2 S}{\p x\p \xi}\right) \right|^{-1} \pmod{S^{m-1}}.
$$
\end{theorem}

\begin{remark}\label{rem Egorov}
From the proof in \cite{MR872855}, it is easy to see that the $\CN_k$-seminorm of the remainder in the above modulo class only relies on the $\CN_{k+2}$-seminorms of $a, b$ and the $\CN_{k+4}$-seminorm of $S$. 
\end{remark}

For our purpose, we need the following version of Egorov's theorem.

\begin{theorem}\label{ThmEgorov2} Let $A\in \OP S^m$ with symbol $a\in S^m$, and $\Phi$ be a classical FIO with amplitude $b\in S^0$ and phase $S$. Let $\CF(x,\xi)=(y,\eta)$ be the canonical transformation associated to $\Phi$. We assume that $\CF$ is a surjective local diffeomorphism of $T^*\TT^d$ with finite inverse branches. Moreover, for each $x\in \TT^d$, the map $\xi\mapsto \CF(x,\xi)$ is bijective.
Then the operator $\Phi^* A \Phi\in \OP S^m$ has a symbol $\wa$ such that
\begin{equation}\label{Egorov symbol}
\wa(y,\eta)=\sum_{\CF(x,\xi)=(y,\eta)} a(x,\xi) |b(x,\xi)|^2 \left|\det \left(\dfrac{\p^2 S}{\p x\p \xi}\right) \right|^{-1} \pmod{S^{m-1}}.
\end{equation}
Moreover, the $\CN_k$-seminorm of the remainder in the above modulo class only relies on the $\CN_{k+2}$-seminorms of $a, b$ and the $\CN_{k+4}$-seminorm of $S$. 
\end{theorem}

\begin{proof} By the properties of the canonical map $\CF$, we can choose an finite open cover $\{\Om_i\}$ of $\TT^d$ such that each $\Om_i\times \RR^d$ is strictly inside an inverse branch of $\CF$.
By partition of unity, there are $\chi_i\in C^\infty_0(\Om_i; [0, 1])$ such that $\sum_i \chi_i=1$. We define symbols $a_i\in S^m(\Om_i\times \RR^d)$ by  $a_i(x,\xi)=\chi_i(x) a(x,\xi)$, and set $A_i=\Op(a_i)$.  
By Theorem \ref{ThmEgorov1}, each $\Phi^* A_i\Phi\in \OP S^m$ has a symbol $\wa_i$ such that 
\begin{equation*}
\wa_i(y,\eta)= \chi_i(x) a(x,\xi) |b(x,\xi)|^2 \left|\det \left(\dfrac{\p^2 S}{\p x\p \xi}\right) \right|^{-1} \pmod{S^{m-1}}
\end{equation*}
for any $(y, \eta)\in \CF(\Om_i\times \RR^d)$ with the only pre-image $(x, \xi)$ in $\Om_i\times \RR^d$; and $\wa_i(y, \eta)=0$ if 
$(y, \eta)\not\in \CF(\Om_i\times \RR^d)$.
Therefore,
\begin{equation*}
\Phi^* A \Phi = \Phi^* (\sum_{i}  A_i) \Phi= \sum_{i} \Phi^* A_i \Phi \in \OP S^m,
\end{equation*}
and its symbol $\wa(y, \eta)$ is given by
\begin{eqnarray*}
\sum_i \wa_i(y, \eta)&=&\sum_{\CF(x,\xi)=(y,\eta)} \sum_{i: x\in \Om_i} \chi_i(x) a(x,\xi) |b(x,\xi)|^2 \left|\det \left(\dfrac{\p^2 S}{\p x\p \xi}\right) \right|^{-1} \pmod{S^{m-1}}\\
&=& \sum_{\CF(x,\xi)=(y,\eta)} a(x,\xi) |b(x,\xi)|^2 \left|\det \left(\dfrac{\p^2 S}{\p x\p \xi}\right) \right|^{-1} \left(\sum_{i: x\in \Om_i} \chi_i(x) \right) \pmod{S^{m-1}} \\
&=& \sum_{\CF(x,\xi)=(y,\eta)} a(x,\xi) |b(x,\xi)|^2 \left|\det \left(\dfrac{\p^2 S}{\p x\p \xi}\right) \right|^{-1} \pmod{S^{m-1}}.
\end{eqnarray*}
The seminorm dependence of the remainder is straightforward by Remark \ref{rem Egorov}.
\end{proof}

\begin{remark}\label{RmkEgorov}
We can easily adapt the proof of the above two theorems 
for the $\hb$-scaled situation and show that 
if $A\in \OP_\hb S^m$ has a symbol $a\in S^m$ 
and $\Phi_\hb$ is the $\hb$-scaled FIO with amplitude $b\in S^0$ and phase $S$,
then the symbol of $\Phi_\hb^* A\Phi_\hb\in \OP_\hb S^m$ is still given by (\ref{Egorov symbol}) but with $\CF(x,\xi)=(y, \eta)$ replaced by $\CF_\hb(x, \hb\xi)=(y, \hb\eta)$, and $\pmod{S^{m-1}}$ replaced by $\pmod{\hb S^{m-1}}$. Moreover, if $\hb r$ is the remainder in the modulo class, then $\CN_k(r)$ only depends on the $\CN_{k+2}$-seminorms of $a, b$ and the $\CN_{k+4}$-seminorm of $S$. 
\end{remark}

\subsubsection{$L^2$-Continuity }

The following result is the classical Calderon-Vaillancourt Theorem (see Theorem 2.8.1 in \cite{MR1872698} for instance).

\begin{theorem}[Calderon-Vaillancourt]
 Let $a(x,\xi)\in S^0$, then $\Op(a): L^2(\TT^d)\to L^2(\TT^d)$ is a bounded operator such that 
$$\|\Op(a)\|_{L^2\to L^2}\le M_1 \|a\|_{C^{k_1}}\le M_1 \CN_{k_1}(a),$$
for some $M_1>0$ and $k_1\in \NN$ that only depend on the dimension $d$.
\end{theorem}

To get finer $L^2$-norm estimates, 
we first state a version of $L^2$-continuity for a classical PDO of order 0 established in \cite{MR2461513}.

\begin{theorem}\label{ThmL2contl}
 If $a(x,\xi)\in S^0$, then $\Op(a):L^2(\TT^d)\to L^2(\TT^d)$ is a bounded operator. Moreover, 
 for any $\ve>0$, there is a decomposition
 $$ \Op(a)=K(\ve)+R(\ve)$$
 such that $K(\ve):L^2(\TT^d)\to L^2(\TT^d)$ is a compact operator and 
 $$ \|R(\ve)\|_{L^2\to L^2}\le \sup_x\limsup_{|\xi|\to \infty} |a(x,\xi)|+\ve=\sup_x\limsup_{|\xi|\to \infty} |a_0(x,\xi)|+\ve, $$
where $a_0\in S^0$ is such that $a=a_0 \pmod{S^{-1}}$.
 \end{theorem}
 
For an $\hb$-scaled PDO of order 0, we need a version of Calderon-Vaillancourt theorem, which applies not only for $\hb\to 0$ but for arbitrary $\hb\in (0, 1]$.
See similar statements in \cite{MR2952218}, Theorem 4.23 or Theorem 5.1 in the formulation of Weyl quantization.

\begin{theorem}\label{ThmL2cont2} If $a(x,\xi)\in S^0$, then $\Op_\hb(a):L^2(\TT^d)\to L^2(\TT^d)$ is a bounded operator. 
Moreover, if $a=a_0+\hb r$ for some $a_0\in S^0$ and $r\in S^{-1}$, then for any $\ve>0$, we have
$$ \|\Op_\hb(a)\|_{L^2\to L^2}\le \sup_{(x,\xi)\in T^*\TT^d} |a_0(x,\xi)| + \ve + \hb C_{k_2}(\ve, a_0, r),
$$
where the constant $C_{k_2}(\ve, a_0, r)$ only depends on $\ve$ and 
the $\CN_{k_2}$-seminorms of $a_0$ and $r$, for some $k_2\in \NN$ that only depends on the dimension $d$.
\end{theorem}

\begin{proof} Recall that $\Op_\hb(a)=\Op(a_\hb)$, 
where $a_\hb(x,\xi)=a(x, \hb\xi; \hb)\in S^0$, by Calderon-Vaillancourt Theorem 
stated above, $\Op_\hb(a)$ is a bounded operator on $L^2(\TT^d)$. 
For the operator norm estimate, 
we mimic the proof in Section 7.5 of \cite{MR2743652}, 
which is originally due to H\"ormander, and which is called ``square root trick". 

For any $\ve>0$, we set $M= \sup_{(x,\xi)\in T^*\TT^d} |a_0(x,\xi)| +\ve$, and $b=\sqrt{M^2-|a_0|^2}\in S^0$. By Theorem \ref{Thm symbol cal hb}, the operator $\Op_\hb(a)^*\Op_\hb(a)\in \OP_\hb S^0$ has a symbol
\begin{eqnarray*}
a^*\# a=\oa a\pmod{\hb S^{-1}} &=& |a_0|^2+2\hb \Re(a_0 r) +\hb^2 |r|^2\pmod{\hb S^{-1}}\\
&=& M^2- b^2\pmod{\hb S^{-1}},
\end{eqnarray*}
that is, $a^*\# a=M^2-b^2+\hb r_1$, for some $r_1\in S^{-1}$. 
Similarly, the operator $\Op_\hb(b)^*\Op_\hb(b)\in \OP_\hb S^0$ has a symbol
$b^*\# b=b^2 \pmod{\hb S^{-1}}$, that is, $b^*\# b=b^2+\hb r_2$, for some $r_2\in S^{-1}$.
Therefore, for any $\varphi\in L^2(\TT^d)$, 
\begin{eqnarray*}
\|\Op_\hb(a) \varphi\|_{L^2}^2&=& \lan \Op_\hb(a)^*\Op_\hb(a)\varphi, \varphi\ran_{L^2} \\
&=& M^2\|\varphi\|_{L^2}^2-\lan \Op_\hb(b^2)\varphi, \varphi\ran_{L^2} +\hb \lan \Op_\hb(r_1)\varphi, \varphi\ran_{L^2} \\
&=& M^2\|\varphi\|_{L^2}^2-\|\Op_\hb(b)\varphi\|_{L^2}^2+\hb \lan \Op_\hb(r_1+r_2)\varphi, \varphi\ran_{L^2} \\
&\le & (M^2 +\hb \|\Op_\hb(r_1+r_2)\|_{L^2\to L^2}) \|\varphi\|_{L^2}^2 \\
&\le & (M +\hb\cdot \frac{\|\Op_\hb(r_1+r_2)\|_{L^2\to L^2}}{2M})^2 \|\varphi\|_{L^2}^2.
\end{eqnarray*}
It remains to show that $\dfrac{\|\Op_\hb(r_1+r_2)\|_{L^2\to L^2}}{2M}$ has an upper bound that is related to 
$\ve$, $a_0$ and $r$. Indeed, by Calderon-Vaillancourt Theorem, 
\begin{eqnarray*}
\|\Op_\hb(r_1+r_2)\|_{L^2\to L^2}=\|\Op((r_1+r_2)_\hb)\|_{L^2\to L^2} 
\le M_1 \left[\CN_{k_1}(r_1)+\CN_{k_1}(r_2)\right].
\end{eqnarray*}
By the construction of $r_1$, we have that 
$ \CN_{k_1}(r_1)$ depends only on the $\CN_{k_1+2}$-seminorms of $a, a_0, r$, and thus only of $a_0, r$,
since $\CN_{k_1+2}(a)=\CN_{k_1+2}(a_0+\hb r)\le \CN_{k_1+2}(a_0)+\CN_{k_1+2}(r)$.
Similarly, $ \CN_{k_1}(r_2)$ depends only on $\ve$ and the $\CN_{k_1+2}$-seminorms of $b$, 
and thus only of $\ve$, $a_0$ and $r$.
In other words, let $k_2=k_1+2$, then 
\begin{equation*}
\dfrac{\|\Op_\hb(r_1+r_2)\|_{L^2\to L^2}}{2M}\le 
\dfrac{M_1 \left[\CN_{k_1}(r_1)+\CN_{k_1}(r_2)\right]}{2\sup_{(x,\xi)} |a_0(x,\xi)|+2\ve}\le C_{k_2}(\ve, a_0, r),
\end{equation*}
where $C_{k_2}(\ve, a_0, r)$ is a constant that only depends on $\ve$ and
the $\CN_{k_2}$-seminorms of $a_0$ and $r$.
This finishes the proof of the theorem.
\end{proof}

\section{Spectral Gap and Coboundary: Proof of the Theorems}

\subsection{Decomposition of Koopman operator}\label{Sec decomp Koopman}

Let $s<0$ be an arbitrary negative order.
Recall that $\CW^s=H^s(\TT^d)\otimes H^{-s}(\TT^\ell)$ is defined by \eqref{fdef space},
and the  Koopman operator $\wF$ acts on $\CW^s$.

We shall decompose the Koopman operator $\wF$
according to fiberwise Fourier expansion. More precisely,
given $\phi\in \CW^s$, 
we write the Fourier series expansion along $\TT^\ell$-direction as
$$ 
\phi(x,y)=\sum_{\bnu\in \ZZ^\ell} \phi_\bnu(x) e^{i2\pi \bnu\cdot y},
$$
where the Fourier coefficients are defined by
$$ 
\phi_\bnu(x)=\int_{\TT^\ell} \phi(x,y)e^{-i2\pi \bnu\cdot y} dy\in H^s(\TT^d),  \quad  \ \bnu\in \ZZ^\ell. $$
It is clear that the family of functions $\{e^{i2\pi \bnu\cdot y}\}_{\bnu\in \ZZ^\ell}$ forms an orthogonal basis of $H^{-s}(\TT^\ell)$, and $\|e^{i2\pi \bnu\cdot y}\|_{H^{-s}(\TT^\ell)}=\lan\bnu\ran^{-s}$.
Therefore, the Hilbert inner product on $\CW^s$ is given by 
\begin{equation*}\label{def: inner product CW}
\lan \phi^1, \phi^2\ran_{\CW^s}=\sum_{\bnu\in\ZZ^\ell} \lan\bnu\ran^{-2s} \lan \phi^1_\bnu, \phi^2_\bnu\ran_{H^s(\TT^d)}, \ \ \text{for any}\ \ \phi^1, \phi^2\in \CW^s.
\end{equation*}
We thus consider the $\lan\bnu\ran^{-1}$-scaled $s$-inner product $\lan \cdot, \cdot\ran_{s, \lan\bnu\ran^{-1}}$ on $H^s(\TT^d)$, or $\lan \cdot, \cdot\ran_{s, \bnu}$ for short. That is, by \eqref{fdef innerprod t scale}, we have for $\varphi, \psi\in H^s(\TT^d)$, 
\begin{equation}\label{def: inner product s nu}
\lan \varphi, \psi\ran_{s, \bnu}= \lan\bnu\ran^{-2s}\lan \varphi, \psi\ran_s=\sum_{\xi\in \ZZ^d} \lan\bnu\ran^{-2s} \lan \xi\ran^{2s}  \wht{\varphi}(\xi)\overline{\wht{\psi}(\xi)},
\end{equation}
and we denote by $H^s_\bnu(\TT^d)$ the space of $s$-order Sobolev functions on $\TT^d$ endowed with the new inner product $\lan\cdot, \cdot\ran_{s, \bnu}$. 
Note that $H^s_\bnu(\TT^d)=H^s(\TT^d)$ as spaces of Sobolev functions, although they are equipped with different but equivalent inner products.
In this way, we obtain an orthogonal decomposition
$$\CW^s=H^s(\TT^d)\otimes H^{-s}(\TT^\ell)\cong \bigoplus_{\bnu \in \ZZ^\ell} H^s_\bnu(\TT^d),$$
such that the inner product of two functions $\phi^j(x,y)=\sum_{\bnu\in \ZZ^\ell} \phi^j_\bnu(x) e^{i2\pi \bnu\cdot y}\in \CW^s$, $j=1, 2$, is given by
\begin{equation*}\label{def: inner product CW nu}
\lan \phi^1, \phi^2\ran_{\CW^s}=\sum_{\bnu\in\ZZ^\ell} \lan \phi^1_\bnu, \phi^2_\bnu\ran_{H^s_\bnu(\TT^d)}.
\end{equation*}
Also this decomposition is $\wF$-invariant,
since for each Fourier mode $\bnu\in \ZZ^\ell$, 
\begin{equation*}
 \wF(\phi_\bnu(x)e^{i2\pi\bnu\cdot y})=[\phi_\bnu(Tx)e^{i2\pi \bnu\cdot \tau(x)}]e^{i2\pi\bnu\cdot y},
\end{equation*}
and it can be shown that $\phi_\bnu(Tx)e^{i2\pi \bnu\cdot \tau(x)}\in H^s(\TT^d)\cong H^s_\bnu(\TT^d)$.\footnote{\ 
This fact is easy to show for $s\in \NN\cup \{0\}$, and hence is also true when $s$ is a negative integer by duality. For the general case, treat $H^s$ as the interpolation between $H^{\lfloor s\rfloor}$ and $H^{\lfloor s\rfloor+1}$. See Section 4.2 in \cite{MR2744150} for details.}
Correspondingly, we decompose $\wF=\bigoplus_{\bnu\in \ZZ^\ell} \wF_\bnu$, where each $\wF_\bnu\cong \wF|H^s_\bnu(\TT^d)$ acts by
\begin{equation}\label{Fnu}
\wF_\bnu\varphi (x)= \varphi(Tx)e^{i2\pi\bnu\cdot \tau(x)}, 
\qquad \  \varphi\in H^s_\bnu(\TT^d).
\end{equation}

Note that $H^s_\bnu(\TT^d)$ is the dual space of $H^{-s}_\bnu(\TT^d)$,
and the dual operator $\wF'_\bnu$ of the operator $\wF_\bnu$ has the form
\begin{equation}\label{RPFnu}
\wF'_\bnu\psi(x)=\sum_{Ty=x} 
\dfrac{e^{i2\pi \bnu\cdot \tau(y)}}{|\Jac(T)(y)|}\psi(y), \qquad 
\psi\in H^{-s}_\bnu(\TT^d).
\end{equation}
In other words, $\wF'_\bnu|H^{-s}_\bnu(\TT^d)$ is the RPF transfer operator over $T:\TT^d\to \TT^d$ for the complex potential function 
$-\log|\Jac(T)| +i2\pi\bnu\cdot \tau$.
In the case when $\bnu=\bzero$, we have that $\wF_\bzero' h=h$, that is, the density function $h(x)$ of the smooth invariant measure $\mu$ w.r.t. $dx$ is provided by the eigenvector 
corresponding to the leading simple eigenvalue 1 of $\wF_\bzero'$. See \cite{MR2129258} for more details.

In the study of the RPF transfer operators, 
we often need to normalize $\wF_\bzero$ into a new operator $\mathcal{L}$ such that $\mathcal{L}1=1$.
To do so, we use the fact $\wF'_\bzero h=h$ in the following particular form: 
\begin{equation*}
\sum_{Ty=x}  \CA(y)=1, \ \ \text{for all}\ x\in \TT^d,
\end{equation*}
where 
\begin{equation}\label{fdef CA}
\CA(y)=\dfrac{1}{|\Jac(T)(y)|}\dfrac{h(y)}{h(Ty)}.
\end{equation}
Then the normalized transfer operator is defined by $\mathcal{L}\psi(x)=\sum_{Ty=x} \CA(y) \psi(y)$.
Similarly, we have for all $n\in \NN$, 
\begin{equation*}
\sum_{T^ny=x}  \CA_n(y)=1, \ \ \text{for all}\ x\in \TT^d,
\end{equation*}
where
\begin{equation}\label{fdef CAn}
\CA_n(y)=\dfrac{1}{|\Jac(T^n)(y)|}\dfrac{h(y)}{h(T^ny)}.
\end{equation}
Then the iterates of $\mathcal{L}$ is given by  $\mathcal{L}^n\psi(x)=\sum_{T^ny=x} \CA_n(y) \psi(y)$.

In this paper, although we do not directly study $\mathcal{L}$, we shall see that
the factors $\CA(y)$ and $\CA_n(y)$ would appear in the formulas of symbols related to the Koopman operators.

\subsection{Spectral gap}

Recall that the notion of spectral gap is given by Definition~\ref{def spec gap}.
According to the decomposition of $\wF:\CW^s\to \CW^s$, 
the spectral gap property for $\wF$ follows from the following propositions. 
The proof of the propositions will be given in the next section,
using semiclassical analysis.

\begin{prop}\label{Pprop1} 
Let $s<0$. There are $C_1>0$ and $\rho_1\in (0,1)$ 
such that for any $\bnu\in \ZZ^\ell$, $\wF_\bnu: H^s_\bnu(\TT^d)\to H^s_\bnu(\TT^d)$ can be written as 
$$ \wF_\bnu=K_\bnu+R_\bnu, $$
where $K_\bnu$ is a compact operator and 
\begin{equation}\label{essential spectrum}
\|R_\bnu^n|H^s_\bnu(\TT^d)\|\le C_1\rho_1^{n}, \ \ n\in \NN.
\end{equation}
\end{prop}

\begin{prop}\label{Pprop2} 
Let $s<0$ and assume that $\tau$ is not an essential coboundary. 
There are $C_2>0$, $\rho_2\in (0,1)$ and $\nu_1>0$ such that for any 
$\bnu\in \ZZ^\ell$ with $|\bnu|\ge \nu_1$, 
\begin{equation*}
\|\wF_\bnu^n|H^s_\bnu(\TT^d)\|\le C_2 \rho_2^n, \ \ n\in \NN.
\end{equation*}
\end{prop}

\begin{remark} 

{\rm (i)}
The quasi-compactness property is well known for Ruelle-Perron-Frobenius transfer operator on H\"older function spaces over expanding maps. Proposition~\ref{Pprop1} can be regarded as its dual version. The estimate in (\ref{essential spectrum}) shows that the essential spectral radius of $\wF_\bnu$ is no more than $\rho_1$. See  (\ref{def: rho1}) for the definition of $\rho_1$, which depends on the Sobolev order $s$ and the minimal expansion rate $\ga$ given by (\ref{fdef gamma});

{\rm (ii)}
Proposition~\ref{Pprop2} shows that the operator $\wF_\bnu$ is essentially a contraction when the Fourier mode $\bnu$ is very large, and the spectral radius of $\wF_\bnu$ is no more than $\rho_2$. See (\ref{def: rho2}) for the definition of $\rho_2$.
\end{remark}

\subsection{Proof of Theorem \ref{ThmSpec gap}}\label{SSproof spec gap}

Recall that the space $\CW^s=H^s(\TT^d)\otimes H^{-s}(\TT^\ell)$ 
is defined in \eqref{fdef space}, where $s<0$. 

\begin{lemma}\label{Lsp rad 1}  
The spectral radius $\Sp(\wF_\bnu|H^s_\bnu(\TT^d))\le 1$ for $\bnu\in \ZZ^\ell$. 
\end{lemma}

\begin{proof} The proof is similar as in \cite{MR2461513}, as we sketch here.

It follows from Proposition~\ref{Pprop1} that 
the essential spectral radius of $\wF_\bnu$ is no more than $\rho_1\in (0, 1)$,
and thus $\wF_\bnu$ has a discrete spectrum outside the circle $\{z\in \CC: |z|=\rho_1\}$.
Let $\Spec(\wF_\bnu)$ the spectrum of the operator $\wF_\bnu$. We can choose 
$\rho_3\in (\rho_1, 1)$ such that $ \Spec(\wF_\bnu)\bigcap \{z: |z|=\rho_3\}=\emptyset$
and $\Spec(\wF_\bnu)\bigcap \{z: |z|>\rho_3\}$ consists of finitely many points.
Let $\Pi_\bnu$ be the spectral projection of $\wF_\bnu$ inside the circle $\{z: |z|=\rho_3\}$, that is,
$$
\Pi_\bnu=\frac{1}{2\pi i} \left.\oint\right._{|z|=\rho_3} (z\mathrm{Id} - \wF_\bnu)^{-1} dz,
$$
then it is well known that $\Pi_\bnu$ is a projection and it commutes with $\wF_\bnu$, i.e., 
$\Pi_\bnu\wF_\bnu=\wF_\bnu\Pi_\bnu$ (See e.g. \cite{MR1335452}). 
Moreover, if we set 
$K^1_\bnu=(\mathrm{Id} -\Pi_\bnu)\wF_\bnu$ and $R^1_\bnu=\Pi_\bnu \wF_\bnu$,
then we have $\wF_\bnu=K^1_\bnu+R^1_\bnu$ such that 
$K^1_\bnu R^1_\bnu=R^1_\bnu K^1_\bnu=0$, and 
\begin{equation*}
\begin{split}
\Spec(K^1_\bnu)&=\Spec(\wF_\bnu)\bigcap \{z: |z|>\rho_3\}, \\
\Spec(R^1_\bnu)&=\Spec(\wF_\bnu)\bigcap \{z: |z|<\rho_3\}.
\end{split}
\end{equation*}
In other words, $K^1_\bnu$ has finite rank and 
the spectral radius of $R^1_\bnu$ is less than $\rho_3$.
To prove that $\Sp(\wF_\bnu|H^s_\bnu(\TT^d))\le 1$ for $\bnu\in \ZZ^\ell$, 
it is then sufficient to show that all eigenvalues of $K^1_\bnu$ are of modulus no more than 1.

The general Jordan decomposition of $K^1_\bnu$ can be written
$$ K^1_\bnu=\sum_{i=1}^k \left(\la_i \sum_{j=1}^{d_i} v_{ij}\otimes w_{ij} + \sum_{j=1}^{d_i-1} v_{ij}\otimes w_{i(j+1)} \right)
$$
where $d_i$ is the dimension of the Jordan block associated with the eigenvalue $\la_i$, with $v_{ij}\in H^s_\bnu(\TT^d)$ and $w_{ij}\in H^{-s}_\bnu(\TT^d)$. We arrange eigenvalues such that $|\la_1|\ge |\la_2|\ge \dots \ge |\la_k|$. 

Now if $|\la_1|>1$, we can choose $\varphi, \psi\in \CDD(\TT^d)$ such that $v_{11}(\psi)\ne 0$ and $w_{11}(\varphi)\ne 0$ since $\CDD(\TT^d)$ is dense in both $H^s_\bnu(\TT^d)$ and $H^{-s}_\bnu(\TT^d)$. On one hand,  
$$ \left|(\psi, \wF_\bnu^n \varphi)_{H^{-s}_\bnu, H^s_\bnu}\right|=\left|\int_{\TT^d} \psi \wF^n_\bnu \varphi dx\right|\le \int |\psi| |\varphi|\circ T^n dx\le  |\psi|_{C^0} |\varphi|_{C^0}.
$$
On the other hand, 
$$ \left|(\psi, \wF_\bnu^n \varphi)_{H^{-s}_\bnu, H^s_\bnu}\right|\ge \left|(\psi, (K^1_\bnu)^n \varphi)_{H^{-s}_\bnu, H^s_\bnu}\right|-\left|(\psi, (R^1_\bnu)^n \varphi)_{H^{-s}_\bnu, H^s_\bnu}\right|.
$$
The second term converges to 0 since $\|(R^1_\bnu)^n|H^s_\bnu(\TT^d)\|=O(\rho_3^n)$, while the first term
\begin{equation*}
 \left|(\psi, (K^1_\bnu)^n \varphi)_{H^{-s}_\bnu, H^s_\bnu}\right|=\left|\sum_{i=1}^k \sum_{r=0}^{\min(n, d_i-1)} {n\choose r} \la_i^{n-r} \sum_{j=1}^{d_i-r} v_{ij}(\psi) w_{i(j+r)}(\varphi)\right|
\end{equation*}
has a leading growth $|\la_1|^n |v_{11}(\psi)| |w_{11}(\varphi)| \to \infty$ as $n\to \infty$, which is a contradiction. Therefore, all eigenvalues of $K^1_\bnu$ are of modulus no more than 1. 
\end{proof}

\begin{lemma}\label{Lsp rad 2} If $\tau$ is not an essential coboundary over $T$, then the spectral radius $\Sp(\wF_\bnu|H^s_\bnu(\TT^d))< 1$ for $\bnu\in \ZZ^\ell\backslash\{\bzero\}$. Moreover, 
1 is the only eigenvalue of $\wF_\bzero$ on the unit circle, which is simple with eigenspace of constant functions.
\end{lemma}

\begin{proof} Note that the essential spectral radius of $\wF_\bnu|H^s_\bnu(\TT^d)$ is no more than $\rho_1$ by \eqref{essential spectrum}. For any $\rho_3\in (\rho_1, 1)$, by Lemma~\ref{Lsp rad 1}, the spectrum of $\wF_\bnu|H^s_\bnu(\TT^d)$ in $\{z\in \CC: \rho_3\le |z|\le 1\}$ consists of finitely many isolated eigenvalues of finite multiplicity. Consequently, the spectral radius $\Sp(\wF_\bnu|H^s_\bnu(\TT^d))$ equals to the largest modulus of its eigenvalues.

Let $\la$ be an eigenvalue of $\wF_\bnu$ with modulus 1, and 
$\varphi\in H^s_\bnu(\TT^d)\subset H^{s-\frac{d}{2}-1}_\bnu(\TT^d)$ be
a corresponding eigenvector such that $\wF_\bnu\varphi=\la \varphi$. 
To prove this lemma, it is sufficient to show that $\bnu=\bzero$, $\la=1$, 
and $\varphi$ is a constant function.

The following argument is essentially due to Pollicott \cite{MR758899}. 
By duality, there is $\psi\in H^{-s+\frac{d}{2}+1}_\bnu(\TT^d)\subset C(\TT^d)$ such that $\wF_\bnu'\psi=\la \psi$. Let $\la=e^{i2\pi c}$ for some $c\in \RR$, and set $\overline{\psi}=\frac{\psi}{h}$, where $h(x)$ is the density function of $\mu$ w.r.t. $dx$. 
By the definition of $\wF_\bnu'$ in (\ref{RPFnu}) and
$\CA(y)$ in (\ref{fdef CA}), we have
\begin{equation}\label{Convex comb}
 \sum_{Ty=x} \CA(y) e^{i2\pi[\bnu\cdot \tau(y)-c]} \overline{\psi}(y) =\overline{\psi}(x).
\end{equation}
Now choose $z$ such that $|\opsi(z)|$ obtains maximum.  Since $\sum_{Ty=z} \CA(y)=1$, we must have $|\opsi(y)|=|\opsi(z)|$ for all $y\in T^{-1}(z)$. By induction, we get that $|\opsi(y)|=|\opsi(z)|$ for all $y\in \bigcup_{n=1}^\infty T^{-n}(z)$. Since $T$ is mixing, the set $\bigcup_{n=1}^\infty T^{-n}(z)$ is dense in $\TT^d$, and hence $|\opsi(x)|=|\opsi(z)|$ is constant. Thus (\ref{Convex comb}) is a convex combination of points of a circle which lies on the circle. From this we deduce that
$$ e^{i2\pi[\bnu\cdot \tau(y)-c]} \opsi(y)=\opsi(Ty)
$$
for all $y\in \TT^d$, and hence (adjust $c$ by an integer value if needed),
$$ \bnu\cdot \tau(y)=c-\dfrac{1}{2\pi} \arg\opsi(y)+\dfrac{1}{2\pi} \arg\opsi(Ty).$$
Since $\tau$ is not an essential coboundary over $T$, we must have $\bnu=\bzero$. By integrating the last equation w.r.t. $d\mu$, we also have that $c=0$ and thus $\la=1$. Further,  $\arg\opsi\equiv$ constant since it is $T$-invariant, and thus $\opsi\equiv$ constant, which implies that $\psi=h\opsi$ is a constant multiple of $h$. Therefore, the space $\{\psi: \wF_\bzero'\psi=\psi\}$ is one-dimensional, so is the space $\{\varphi: \wF_\bzero\varphi=\varphi\}$ by duality.
Since $\wF_\bzero 1= 1$, we must have that $\varphi$ is a constant function.
\end{proof}

Now we are ready to prove Theorem \ref{ThmSpec gap}.

\begin{proof}[Proof of Theorem \ref{ThmSpec gap}] 
We assume that $\tau$ is not an essential coboundary.  Hence, the results
of Proposition~\ref{Pprop2} can be applied.

By Lemma~\ref{Lsp rad 1} and \ref{Lsp rad 2}, 
we have the following:
\begin{enumerate}
\item[(i)] When $\bnu=\bzero$, the spectrum $\Spec(\wF_\bzero)=\{1\}\cup \CK_\bzero$, where 1 is a simple eigenvalue of $\wF_\bzero$ and $\CK_\bzero$ is a compact subset of the open unit disk $\DD$. 
Choose $r_\bzero\in [0, 1)$ such that $\CK_\bzero$ is strictly contained inside the circle $\{z: |z|=r_\bzero\}$.
Let $\Pi_\bzero$ be the spectral projection of $\wF_\bzero$ inside $\{z: |z|=r_\bzero\}$,
and note that $\ker(\Pi_\bzero)$ is the one-dimensional subspace consisting of constant functions.
Write  $V_\Const=\ker(\Pi_\bzero)$ and $V_\bzero=\mathrm{Im}(\Pi_\bzero)$.
It is clear that  
\begin{equation*}
H^s_\bzero(\TT^d)=V_\Const \oplus V_\bzero,
\end{equation*}
Moreover, both subspaces are preserved by $\wF_\bzero$ 
since $\Pi_\bzero\wF_\bzero=\wF_\bzero\Pi_\bzero$, and 
the spectra of $\wF_\bzero$ restricted to these two subspaces are given by
\begin{equation*}
\Spec(\wF_\bzero|_{V_\Const}) = \{1\},  \ \text{and} \ \ 
\Spec(\wF_\bzero|_{V_\bzero})  = \CK_\bzero\subset \{z: |z|<r_\bzero\}.
\end{equation*}
Therefore, there is $M_\bzero>0$
such that $\|\wF_\bzero^n|V_\bzero\|\le M_\bzero r_\bzero^n$ for all $n\in \NN$;
\item[(ii)] For all $\bnu\in \ZZ^\ell\backslash\{\bzero\}$, the spectrum $\Spec(\wF_\bnu)$ is strictly inside the open unit disk $\DD$. Equivalently, there are $M_\bnu>0$ and $r_\bnu\in [0, 1)$ such that $\|\wF_\bnu^n|H^s_\bnu(\TT^d)\|\le M_\bnu r_\bnu^n$ for all $n\in \NN$.
\end{enumerate}
And by Proposition~\ref{Pprop2},
\begin{enumerate}
\item[(iii)] When $|\bnu|\ge \nu_1$, $\|\wF_\bnu^n|H^s_\bnu(\TT^d)\|\le C_2\rho_2^n$ for all $n\in \NN$.
\end{enumerate}
We set the direct sum
\begin{equation}\label{def: V}
V=V_\bzero\oplus \left(\bigoplus\limits_{\bnu\in \ZZ^\ell\backslash\{\bzero\}} H^s_\bnu(\TT^d) \right),
\end{equation}
then $\CW^s=V_\Const \oplus V$. Furthermore, let $C_4:=\max\{C_2, \max_{|\bnu|<\nu_1}\{M_\bnu\}\}$ and $\rho_4:=\max\{\rho_2, \max_{|\bnu|<\nu_1}\{r_\bnu\}\}$, then we have $\|\wF^n|V\|\le C_4\rho_4^n$ for all $n\in \NN$. In other words,
$\wF=\oplus_{\bnu\in \ZZ^\ell} \wF_\bnu$ has spectrum
$$ 
\Spec(\wF)=\{1\}\cup \CK,
$$
where $\CK=\Spec(\wF|V)\subset \{z\in \CC: |z|\le \rho_4\}$, and $1$ is the only leading simple eigenvalue with eigenvectors being constant functions. So $\wF:\CW^s\to \CW^s$ has spectral gap.
\end{proof}

\subsection{Proof of Theorem~\ref{ThmMain thm}}

Now we use Theorem~\ref{ThmSpec gap} to prove Theorem~\ref{ThmMain thm}.
What we need to do is to show that if $\wF: \CW^s\to \CW^s$ has a spectral gap,
then it is exponentially mixing.  In the proof we regard the H\"older continuous observables 
$\phi$ and $\psi$ as elements in $\CW^s$ and $\CW^{-s}$ respectivly.

\begin{proof}[Proof of Theorem \ref{ThmMain thm}] 
Since $\wF: \CW^s\to \CW^s$ has a spectral gap, we can write
$$\wF=\wF|V_\Const + \wF|V=:\CP+\CN,$$ 
where $V$ is defined in (\ref{def: V}).
From the proof of Theorem \ref{ThmSpec gap}, we know that 
\begin{enumerate}[(a)]
\item $\CP$ is a 1-dimensional projection, i.e., $\CP^2=\CP$ ;
\item $\CN$ is a bounded operator with spectral radius $\Sp(\CN)\le \rho_4<1$. In fact, $\|\CN^n\|\le C_4\rho_4^n$ for all $n\in \NN$;
\item $\CP\CN=\CN\CP=0$.
\end{enumerate}
Furthermore, 1 is the greatest simple eigenvalue for $\wF_\bzero$ with eigenvector 1 as well as for $\wF_\bzero'$ with eigenvector $h$.
It means that the bilinear form associated to $\CP$ on $\CW^{-s}\times \CW^{s}$
is generated by $1\otimes h\in \CW^s\otimes \CW^{-s}$, that is,
for any $\psi\in \CW^{-s}$ and $\phi\in \CW^{s}$, 
\begin{equation*}
(\psi, \CP(\phi))_{\CW^{-s}, \CW^s}=(\psi, (1\otimes h)(\phi))_{\CW^{-s}, \CW^s}.
\end{equation*}

Suppose $\phi, \psi\in C^\al(\TT^{d+\ell})$ are given.  
Pick $s\in [-\al, 0)$ and let 
$\CW^{-s}=H^{-s}(\TT^d)\otimes H^{s}(\TT^\ell)$.  
Then the dual space of $\CW^{-s}$ is
$\CW^{s}=H^{s}(\TT^d)\otimes H^{-s}(\TT^\ell)$.   
Note that 
$C^\al(\TT^{d+\ell})$ is contained in both $\CW^{-s}$ and $\CW^{s}$, and thus 
$\phi\in \CW^s$ and $\psi h\in \CW^{-s}$, where $h\in C^\infty(\TT^d)$ 
is the density function of $\mu$ w.r.t. $dx$.
Hence, 
\begin{equation*}
\begin{split}
\int (\phi\circ F^n) \psi dA  
=& \int (\phi\circ F^n) \psi h\ dxdy \\
=& (\psi h, \wF^n(\phi))_{\CW^{-s}, \CW^s} \\
=& (\psi h, \CP(\phi))_{\CW^{-s}, \CW^s} + (\psi h, \CN^n(\phi))_{\CW^{-s}, \CW^s}\\
=& (\psi h, (1\otimes h) (\phi) )_{\CW^{-s}, \CW^s} + (\psi h, \CN^n(\phi))_{\CW^{-s}, \CW^s}\\
=&  \left(\psi h, \left(\int \phi h dxdy\right)\cdot 1 \right)_{\CW^{-s}, \CW^s} + (\psi h, \CN^n(\phi))_{\CW^{-s}, \CW^s} \\
=&  \int \psi dA \int \phi dA + (\psi h, \CN^n(\phi))_{\CW^{-s}, \CW^s}.
\end{split}
\end{equation*}
That is, the correlation function
$$ C_n(\phi, \psi; F, dA)=|(\psi h, \CN^n(\phi))_{\CW^{-s}, \CW^s}|\le \|\CN^n\| \|\psi h\|_{\CW^{-s}} \|\phi\|_{\CW^s}\le C_{\phi, \psi} \rho_4^n$$
where $C_{\phi, \psi}=C_4\|\psi h\|_{\CW^{-s}} \|\phi\|_{\CW^s}.$
\end{proof}

\begin{remark} Using some Sobolev inequalities, it is not hard to get that $\|\psi h\|_{\CW^{-s}}\le C_5\|\psi\|_{C^\al}\|h\|_{C^1}$ and $\|\phi\|_{\CW^s}\le C_6\|\phi\|_{C^\al}$, and hence $C_{\phi, \psi}\le C_7\|\phi\|_{C^\al} \|\psi\|_{C^\al}$.
\end{remark}

\subsection{Proof of Theorem~\ref{ThmCobd}}

Now we show the characters of the non-mixing skew products $F_\tau$, that is,
$\tau_1, \tau_2,\dots, \tau_\ell$ are integrally dependent $\mmod$ $\fB$. 

\begin{proof}[Proof of Theorem~\ref{ThmCobd}]

(i)$\Rightarrow$(ii).  
Suppose $\tau_1, \tau_2, \dots, \tau_\ell$ are integrally dependent $\mmod$~$\fB$, that is, there are $v\in \ZZ^\ell\backslash\{\bzero\}$, $c\in \RR$ and $u\in C^\infty(\TT^d, \RR)$ such that 
$$ v\cdot \tau(x)=c+u(x)-u(Tx). $$

For any $(x,y)\in \TT^d\times \TT^\ell$, the set 
\begin{equation*}
\CL(x,y)=\{(x',y')\in \TT^d\times \TT^\ell: v\cdot y'+u(x')= v\cdot y+u(x) \pmod \ZZ\}.
\end{equation*}
is well-defined. Moreover, 
since $u$ is a smooth map, $\CL(x,y)$ is a smooth $(d+\ell-1)$-dimensional 
manifold, and $\{\CL(x,y): \ (x,y)\in \TT^d\times \TT^\ell\}$
form a foliation of $\TT^d\times \TT^\ell$.
It is clear that for any fixed $x\in \TT^d$,
\[
\CL(x,y)|_{\{x\}\times \TT^\ell}
=\{(x,y')\in \{x\}\times \TT^\ell: v\cdot (y-y')=0 \pmod \ZZ\}.
\]
It implies that the leaves of $\CL(x,y)|_{\{x\}\times \TT^\ell}$ are normal 
to $v$.

For $(x',y')\in \CL(x,y)$, the definition of $F$ gives 
\begin{equation*}
F(x,y)=(Tx, y+\tau(x)) \quad \text{and} \quad
F(x',y')=(Tx', y'+\tau(x')),
\end{equation*}
then
\begin{equation*}
v\cdot (y+\tau(x))+u(Tx)
=v\cdot y+v\cdot \tau(x)+u(Tx)
=v\cdot y+c+u(x) \pmod \ZZ
\end{equation*}
and similarly $v\cdot (y'+\tau(x'))+u(Tx')=v\cdot y'+c+u(x') \pmod \ZZ$.
Hence we obtain
\begin{equation*}
v\cdot (y'+\tau(x'))+u(Tx')=v\cdot (y+\tau(x))+u(Tx) \pmod \ZZ.
\end{equation*}
By definition of $\CL$, we get $F(x',y')\in \CL(F(x,y))$, that is, the foliation is $F$-invariant.


(ii)$\Rightarrow$(iii).  
Restricted to $\{p\}\times \TT^\ell$ the leaves of the foliation $\CL$
become $(\ell -1)$ dimensional tori because the leaves are normal to $v$.
Hence the quotient space $\TT^\ell/\sim$ is a circle $\TT$, where
$y\sim y'$ if $y$ and $y'$ are in the same leave of 
$\CL|_{\{p\}\times \TT^\ell}$.
Let $\pi|_{\{p\}\times \TT^\ell}: \{p\}\times \TT^\ell\to \{p\}\times\TT$ 
be the quotient map. Clearly $\pi|_{\{p\}\times \TT}$ is continuous. 
Since $F: \{p\}\times \TT^\ell \to \{p\}\times \TT^\ell$ is a rotation 
given by $F(p, y)=(p, y+\tau(p))$ and preserves the leaves,
it induces an rotation 
$G|_{\{p\}\times \TT}: \{p\}\times\TT\to \{p\}\times\TT$
such that $\pi|_{\{p\}\times \TT}\circ F|_{\{p\}\times \TT}
=G|_{\{p\}\times \TT}\circ \pi|_{\{p\}\times \TT}$.   
We also denote by $R_c$ the rotation, where $c\in \TT$.  

$\pi|_{\{p\}\times \TT}$ and $G|_{\{p\}\times \TT}$ can be extended to map 
$\pi: \TT^d\times \TT^\ell \to \TT^d\times \TT$
and $G: \TT^d\times \TT \to \TT^d\times \TT$ in a natural way.
That is, for any $(x,y)\in \TT^d\times \TT^\ell$,
let $y'\in \CL(x,y)\cap (\{p\}\times \TT^d)$, and define 
$\pi(x,y)=(x, \pi|_{\{p\}\times \TT^\ell}(y'))$; 
and for any $(x, \bar y)\in \TT^d\times \TT$
define $G(x,\bar y)=(Tx, G_{\{p\}\times \TT}(\bar y))=(Tx, R_c(\bar y))$.
Note that for all $x\in \TT^d$,
the leaves of $\CL(x,y)|_{\{x\}\times \TT^\ell}$ are normal to $v$.
It is easy to check that $\pi\circ F=G\circ \pi$.

(iii)$\Rightarrow$(iv).  
Weak mixing property does not hold for the circle rotation $R_c$, let alone the extension $F$.

(iv)$\Rightarrow$(i).  
It follows from the results by Parry and Pollicott \cite{MR1632190}, and also by Field and Parry \cite{MR1644099}. 
\end{proof}

\section{Spectrums of $\wF_\nu$: 
Proof of Proposition~\ref{Pprop1} and \ref{Pprop2}}

In this section, we shall prove the two main propositions - Proposition~\ref{Pprop1} and Proposition~\ref{Pprop2}.
The main scheme of the proofs and constructions are originated from Faure \cite{MR2785978} 
and many other related works. More precisely,
we shall use semiclassical analysis to prove Proposition~\ref{Pprop1}
and Proposition~\ref{Pprop2}. The flexibility of the parameter $\hb$ allows us to deal with the operators 
$\wF_\bnu: H^s_\bnu(\TT^d)\to H^s_\bnu(\TT^d)$, $\bnu\in \ZZ^d$, in two different ways.
To be precise, for any fixed frequency $\bnu\in \ZZ^\ell$, we take $\hb=1$ and treat $\hF_\bnu$ as a classical FIO (up to a smoothing operator)
in the proof of Proposition~\ref{Pprop1};
while for Proposition~\ref{Pprop2}, we set $\hb=1/\max\{1, |\bnu|\}$ and regard $\hF_\bnu$ as 
an $\hb$-scaled FIO (up to an $\hb$-scaled smoothing operator).

\subsection{The Sobolev spaces with non-standard inner products}\label{SSsp op} 

We first construct a particular symbol on $T^*\TT^d$ as follows. Choose 
\begin{equation}\label{fdef R}
R>\max\left\{1, \frac{\max\{1, \;\  2\|D\tau\|\} }{\ga-1}\right\},
\end{equation}
where $\ga$ is given in (\ref{fdef gamma}), and $\|D\tau\|=\sup\limits_{x\in \TT^d} |D_x\tau|$.
Let $g_0\in C^\infty(\RR^+)$ be such that
\begin{equation}\label{fdef g}
g_0(t)=\begin{cases} 1, \ & \ t\le R; \\ 
    t, \ & \ t\ge \frac{\ga+1}{2}R, \end{cases}
\end{equation}
and for $t\in [R, \frac{\ga+1}{2}R)$, $g_0(t)$ is strictly increasing and $1\le g_0(t)\le t$. Set $g(\xi)=g_0(|\xi|)$ for $\xi\in \RR^d$. 
Given $s<0$, define an elliptic symbol
\begin{equation}\label{fdef las}
\la_{s}(x,\xi)=h(x)^{\frac12} g\left(\xi\right)^s \in S^s,
\end{equation}
where $h(x)$ is the density function of $\mu$ w.r.t. $dx$.
Further, given
$\bnu\in\ZZ^\ell$, define 
\begin{equation}\label{fdef las nu}
\la_{s,\bnu}(x,\xi)=\la_s\left(x, \frac{\xi}{\ldbrack\bnu\rdbrack}\right) \in S^s,
\end{equation}
where $\ldbrack\bnu\rdbrack:=\max\{1, |\bnu|\}$.

Denote $\La_{s, \bnu}=\Op(\la_{s, \bnu})\in \OP S^s$, and 
define an inner product on $H^s(\TT^d)$ by 
\begin{equation*}
\lan \varphi, \psi\ran_{\La_{s, \bnu}}=\lan \La_{s,\bnu} \varphi, \La_{s,\bnu} \psi\ran_{L^2}, \ \ \ \ \ \varphi, \psi\in H^s(\TT^d).
\end{equation*}
When equipped with $\lan\cdot, \cdot\ran_{\La_{s,\bnu}}$, $H^s(\TT^d)$ is denoted by $H_{\La_{s, \bnu}}(\TT^d)$ instead.
The Sobolev space $H_{\La_{s,\bnu}}$ is unitarily equivalent 
to the $L^2$ space, that is,
\begin{equation*}
\La_{s,\bnu} H_{\La_{s,\bnu}}(\TT^d)\cong L^2(\TT^d), \ \ 
\text{or}\ \ H_{\La_{s,\bnu}}(\TT^d)\cong \La_{s, \bnu}^{-1} L^2(\TT^d). 
\end{equation*}

We claim that the spaces  $H_{\La_{s,\bnu}}(\TT^d)$ and $H^s_\bnu(\TT^d)$, which are identical as the set of $s$-order Sobolev functions, have comparable inner products in the following sense:
there is $C_1=C_1(d, s)>0$ such that 
\begin{equation}\label{choose: C1}
\frac{1}{C_1}\left|\lan \varphi, \psi\ran_{{s, \bnu}}\right| \le \left|\lan \varphi, \psi\ran_{\La_{s, \bnu}}\right| \le C_1\left|\lan \varphi, \psi\ran_{{s, \bnu}}\right|  , \ \ \text{for any}\ \varphi, \psi\in H^s(\TT^d).
\end{equation}
To see this, 
recall that $H^s_\bnu(\TT^d)$ is equipped with the $\lan\bnu\ran^{-1}$-scaled $s$-inner product $\lan \cdot, \cdot\ran_{s, \bnu}$ given by (\ref{def: inner product s nu}).
Alternatively, we have
$$
\Op(\lan\bnu\ran^{-s}\lan\xi\ran^s) H^s_\bnu(\TT^d)\cong L^2(\TT^d).
$$
Then the comparability of inner products simply follows from that $\la_{s, \bnu}(x, \xi)\asymp \lan\bnu\ran^{-s} \lan\xi\ran^s$, i.e., there is $C_0=C_0(d, s)>0$ such that for any $\bnu \in \ZZ^\ell$, 
\begin{equation*}
\frac{1}{C_0}\le \dfrac{\la_{s, \bnu}(x, \xi)}{ \lan\bnu\ran^{-s}\lan\xi\ran^s}\le C_0, \ \ \text{for any}\ (x,\xi)\in T^*\TT^d.
\end{equation*}

\subsection{Proof of Proposition~\ref{Pprop1}} 

Recall that $\wF_\bnu: H^s_\bnu(\TT^d)\to H^s_\bnu(\TT^d)$ is defined in \eqref{Fnu}. Switching to the inner product $\lan\cdot, \cdot\ran_{\La_{s,\bnu}}$, we mainly study the operator $\wF_\bnu: H_{\La_{s,\bnu}}(\TT^d)\to H_{\La_{s,\bnu}}(\TT^d)$ instead.

\begin{proof}[Proof of Proposition~\ref{Pprop1}]
Let $s<0$ and $\bnu\in \ZZ^\ell$ be fixed. By the formula of $\wF_\bnu$ in \eqref{Fnu}, the Fourier transform \eqref{fdefFTransf} and inverse transform \eqref{fdefinvFTransf}, we rewrite
\begin{equation*}
\wF_\bnu \varphi (x)
=\sum_{\xi\in \ZZ^d}\int_{\TT^d} e^{i2\pi\bnu\cdot \tau(x)}e^{i2\pi[Tx\cdot \xi - y\cdot \xi]} 
\varphi(y) dy.
\end{equation*}  
In the above formula, $\wF_\bnu$ defines a classical (i.e., $\hb=1$) toroidal Fourier series operator (FSO)
with amplitude $a^\bnu(x,\xi)=e^{i2\pi\bnu\cdot \tau(x)}\in S^0$ and phase $S(x,\xi)=Tx\cdot \xi\in S^1$. 
For more details on FSO, we refer the readers to \cite{MR2567604}, Section 4.13.

Let $\Phi_1(a^\bnu, S)$ be the classical toroidal FIO given by Definition~\ref{DefhFIO} 
with the same amplitude and phase. It is shown in \cite{MR2567604} that 
the difference operator $\wF_\bnu-\Phi_1(a^\bnu, S)$ is smoothing and thus compact. 
Therefore, the result in Proposition~\ref{Pprop1} holds for $\wF_\bnu$ if and only if
it holds for $\Phi_1(a^\bnu, S)$.
In the rest of the proof, we shall analyze the classical toroidal FIO $\Phi_1(a^\bnu, S)$,
but still denote it by $\wF_\bnu$ for notational convenience.

The canonical transformation $\CF: (x,\xi)\mapsto (y,\eta)$ associated to 
$\hF_\bnu$ is given by
\begin{equation*}
y=Tx, \ \ \ \ \eta=[(D_xT)^t]^{-1} \xi,
\end{equation*}
which is irrelevant to $\bnu$ since the phase function $S$ is independent of $\bnu$.

We have the following commutative diagram
\begin{displaymath}
\begin{CD}
H_{\La_{s,\bnu}}(\TT^d) @> \hF_\bnu >>  H_{\La_{s,\bnu}}(\TT^d)\\
@ V \La_{s,\bnu} VV  @VV \La_{s, \bnu} V\\
L^2(\TT^d) @> Q_\bnu >> L^2(\TT^d)\ ,
\end{CD}
\end{displaymath}
where $Q_\bnu=\La_{s, \bnu} \hF_\bnu \La_{s, \bnu}^{-1}$.  We then set 
\begin{equation*}
\begin{split}
P_\bnu=&Q_\bnu^*Q_\bnu
= (\La_{s, \bnu}^{-1})^* \left[\hF_\bnu^* (\La_{s, \bnu}^*\La_{s, \bnu}) \hF_\bnu\right] \La_{s, \bnu}^{-1}\\
=& (\Op(\la_{s,\bnu})^{-1})^* \Big[\Phi(a^\bnu, S)^* \big(\Op(\la_{s,\bnu})^*\Op(\la_{s,\bnu})\big) \Phi(a^\bnu, S)\Big] \Op(\la_{s,\bnu})^{-1}.
\end{split}
\end{equation*}

By the symbol calculus (Theorem~\ref{Thm symbol cal}) and the Egorov's theorem (Theorem~\ref{ThmEgorov2}), the operator $P_\bnu$ is a classical PDO of order 0. Denote by $p_\bnu(x,\xi)$ the symbol of $P_\bnu$.
By the $L^2$-continuity theorem (Theorem~\ref{ThmL2contl}),
$P_\bnu:L^2(\TT^d)\to L^2(\TT^d)$ is a bounded operator such that for any $\ve>0$, we can write 
$P_\bnu=K_\bnu^0(\ve)+R_\bnu^0(\ve)=:K_\bnu^0+R_\bnu^0$, where $K_\bnu^0$ is compact;
and moreover, by Lemma~\ref{LPmu symbol1} below, the definition of $g$ in Section~\ref{SSsp op} and the definition $\ga$ in \eqref{fdef gamma}, we get 
\begin{equation*}
\begin{split}
\|R_\bnu^0\|_{L^2\to L^2} 
\le &\ \sup_x\limsup_{|\xi|\to \infty} |p_\bnu(x,\xi)|+\ve \\
 =& \sup_x \limsup_{|\xi|\to \infty} \sum_{x=Ty} \CA(y) \left(\dfrac{g((D_yT)^t(\xi/\ldbrack\bnu\rdbrack))}{g(\xi/\ldbrack\bnu\rdbrack)}\right)^{2s} + \ve  \\
\le&  \sup_x  \sum_{x=Ty} \CA(y) \limsup_{|\xi|\to \infty}  \left(\dfrac{|(D_yT)^t\xi|}{|\xi|}\right)^{2s} + \ve \\
\le & \sup_x  \sum_{x=Ty} \CA(y) \ga^{2s} +\ve = \ga^{2s} +\ve.
\end{split}
\end{equation*}
Choose $\ve>0$ small enough such that 
\begin{equation}\label{def: rho1}
\rho_1:=\sqrt{\ga^{2s} +\ve}<1.
\end{equation} 
By the polar decomposition, $Q_\bnu=\sqrt{P_\bnu} U_\bnu$ for some partial isometry $U_\bnu:L^2(\TT^d) \to L^2(\TT^d)$, and thus there is also a decomposition $Q_\bnu=K_\bnu^1+R_\bnu^1$ such that $K_\bnu^1$ is compact and $\|R_\bnu^1\|_{L^2\to L^2}\le \rho_1$. By unitary equivalence between $Q_\bnu: L^2(\TT^d)\to L^2(\TT^d)$ and $\hF_\bnu: H_{\La_{s,\bnu}}(\TT^d)\to H_{\La_{s,\bnu}}(\TT^d)$, 
there is a similar decomposition 
$\hF_\bnu=K_\bnu+R_\bnu$ such that $K_\bnu$ is compact, 
and $\|R_\bnu|H_{\La_{s,\bnu}}(\TT^d)\|\le \rho_1$. 

By the choice of the constant $C_1$ in (\ref{choose: C1}),
we get 
$$  
\|R^n_\bnu|H^s_\bnu(\TT^d)\| 
\le C_1 \|R^n_\bnu|H_{\La_{s,\bnu}}(\TT^d)\|
\le C_1 \|R_\bnu|H_{\La_{s,\bnu}}(\TT^d)\|^n\le C_1\rho_1^n.
$$
This completes the proof of Proposition~\ref{Pprop1}.
\end{proof}

\begin{lemma}\label{LPmu symbol1} $P_\bnu\in \OP S^0$ has a symbol 
\begin{equation*}
p_\bnu(x,\xi)=\sum_{x=Ty} \CA(y) \left(\dfrac{g((D_yT)^t(\xi/\ldbrack\bnu\rdbrack))}{g(\xi/\ldbrack\bnu\rdbrack)}\right)^{2s} \pmod {S^{-1}},
\end{equation*}
where $\CA(y)$ is defined in (\ref{fdef CA}). 
\end{lemma}

\begin{proof}
Note that $\La_{s, \bnu}\in \OP S^s$ has a symbol $\la_{s, \bnu}$ given by (\ref{fdef las nu}).
By Theorem~\ref{Thm symbol cal}, 
$\La_{s,\bnu}^*\in \OP S^s$ has a symbol $\la_{s, \bnu} \pmod{S^{s-1}}$; 
$\La_{s,\bnu}^{-1}\in \OP S^{-s}$ has a symbol 
$\la_{s, \bnu}^{-1} \pmod{S^{-s-1}}$, and so does  $(\La_{s,\bnu}^{-1})^*\in \OP S^{-s}$.
Further, $\La_{s,\bnu}^*\La_{s,\bnu}\in \OP S^{2s}$ has a symbol $\la_{s,\bnu}^2 \pmod{S^{2s-1}}$. 
Then by Egorov's theorem \ref{ThmEgorov2}, 
$\hF_\bnu^* (\La_{s,\bnu}^*\La_{s,\bnu}) \hF_\bnu\in \OP S^{2s}$ has a symbol 
\begin{eqnarray*}
\wa(y,\eta)
&=&\sum_{\substack{y=Tx, \\ \eta=[(D_xT)^t]^{-1} \xi}} \la_{s,\bnu}^2(x,\xi) \left| e^{i2\pi \bnu\cdot \tau(x)} \right|^2 |\det(D_xT)^t|^{-1} \pmod{S^{2s-1}}\\
\\
&=& \sum_{y=Tx} \dfrac{\la_{s,\bnu}^2(x, (D_xT)^t\eta )}{|\Jac(T)(x)|} \pmod{S^{2s-1}}.
\end{eqnarray*}
Use the composition rule and recall the definition of $\la_{s, \bnu}$ in (\ref{fdef las nu}), we have $P_\bnu\in \OP S^0$ with a symbol 
\begin{eqnarray*}
p_\bnu(y,\eta)
&=&\sum_{y=Tx} \dfrac{\la_{s,\bnu}^2(x, (D_xT)^t\eta)} 
     {|\Jac(T)(x)|} \dfrac{1}{\la_{s,\bnu}^2(y,\eta)} \pmod{S^{-1}}\\
&=& \sum_{y=Tx} \dfrac{1}{|\Jac(T)(x)|} \dfrac{h(x)}{h(y)} \dfrac{g((D_xT)^t(\eta/\ldbrack\bnu\rdbrack))^{2s}}{g(\eta/\ldbrack\bnu\rdbrack)^{2s}}\pmod{S^{-1}} \\
&=& \sum_{y=Tx} \CA(x) \left(\dfrac{g((D_xT)^t(\eta/\ldbrack\bnu\rdbrack))}{g(\eta/\ldbrack\bnu\rdbrack)}\right)^{2s} \pmod{S^{-1}}.
\end{eqnarray*}
This is what we need.
\end{proof}

\subsection{Proof of Proposition~\ref{Pprop2}}

To prove Proposition~\ref{Pprop2}, we relate the semiclassical parameter $\hb$ with a given frequency  
$\bnu\in \ZZ^\ell$ by setting $\hb=1/\ldbrack\bnu\rdbrack$. In this way, 
we study the operator $\hF_\bnu^n$ as an $\hb$-scaled toroidal FSO, and hence separate the
dependence of $\hF_\bnu^n$ on the frequency $\bnu$ into two parts: the dependence on the modulus
$\ldbrack\bnu\rdbrack=1/\hb$ and that on the direction vector $\bn_\bnu:=\bnu/\ldbrack\bnu\rdbrack$.
The key step of the proof is the estimate stated and proved 
in Lemma~\ref{Lmain estmt} in the next section.

\begin{proof}[Proof of Proposition~\ref{Pprop2}]
Let $s<0$ and assume that $\tau$ is not an essential coboundary 
over $T:\TT^d\to \TT^d$.
Given $\bnu\in \ZZ^\ell$, set $\hb=1/\ldbrack\bnu\rdbrack$, 
where $\ldbrack\bnu\rdbrack=\max\{1, |\bnu|\}$. 
Also denote by $\bn_\bnu=\bnu/\ldb \bnu\rdb$ the direction vector of $\bnu$, and note that either $\bn_\bnu=\bzero$ (only if $\bnu=\bzero$) or $\bn_\bnu$ lies on the $(\ell-1)$-dimensional unit sphere  $\SSS^{\ell-1}$.

For any $n\in \NN$, 
the operator $\wF_\bnu^n$ can then be rewritten as
\begin{equation*}
\begin{split}
\wF_\bnu^n \varphi(x) 
=& \varphi(T^n x)e^{i2\pi \bnu\cdot \sum_{k=0}^{n-1} \tau(T^kx) } \\
=& \sum_{\xi\in \ZZ^d} \int_{\TT^d} e^{i 2\pi \{ [T^nx\cdot \xi+ \bnu
  \cdot  \sum_{k=0}^{n-1} \tau(T^kx)] - y\cdot \xi \}} \varphi(y) dy\\
=& \sum_{\xi\in \hb\ZZ^d} \int_{\TT^d} e^{i 2\pi \frac{1}{\hb} \{ [T^nx\cdot \xi+ \bn_\bnu
  \cdot  \sum_{k=0}^{n-1} \tau(T^kx)] - y\cdot \xi \}} \varphi(y) dy.
\end{split}
\end{equation*}
That is, $\wF_\bnu^n$ is regarded as an $\hb$-scaled toroidal Fourier series operator (FSO) 
with amplitude $a= 1$ and phase 
\begin{equation}\label{def phase nu n}
S_{\bn_\bnu,  n}(x,\xi)
=T^nx\cdot \xi+ \bn_\bnu \cdot  \sum_{k=0}^{n-1} \tau(T^kx).  
\end{equation}
Please see \cite{MR2567604}, Section 4.13 for details about FSO. 

Let $\Phi_\hb(1, S_{\bn_\bnu,  n})$ be the $\hb$-scaled toroidal FIO given by Definition~\ref{DefhFIO}.
It can be shown that the difference $\wF_\bnu^n-\Phi_\hb(1, S_{\bn_\bnu,  n})=\hb R$ for some smoothing operator $R$. 
Therefore, as $\hb=1/\ldbrack\bnu\rdbrack\to 0$, 
the result in Proposition~\ref{Pprop2} holds for $\wF_\bnu^n$ if and only if
it holds for $\Phi_\hb(1, S_{\bn_\bnu, n})$.
In the rest of the proof, we shall analyze the $\hb$-scaled toroidal FIO $\Phi_\hb(1, S_{\bn_\bnu,  n})$,
but still denote it by $\wF_\bnu^n$ for notational convenience.

Note that the canonical transformation 
$\CF_\hb: (x, \hb\xi)\mapsto (y, \hb\eta)$ associated to $\hF_{\bnu}^n$ 
is given by
\begin{equation*}
y=T^nx, \ \ \ \ \eta=[(D_xT^n)^t]^{-1} [\xi-W_{\bn_\bnu,n}(x)],
\end{equation*}
where 
\begin{equation}\label{fdef Wnun}
W_{\bn_\bnu, n}(x)=W_n(x)\bn_\bnu \ \ \text{and}\ \ \ W_n(x)=\sum_{k=0}^{n-1} (D_xT^k)^t (D_{T^k x}\tau)^t.
\end{equation}

By \eqref{fdef las}, \eqref{fdef las nu} and that $\hb=1/\ldb\bnu\rdb$, we rewrite $\la_{s,\bnu}(x,\xi)=\la_s(x, \hb \xi)$, and hence $\La_{s,\bnu}=\Op(\la_{s,\bnu})=\Op_\hb(\la_s)\in \OP_\hb S^s$.
The following commutative diagram
\begin{displaymath}
\begin{CD}
H_{\La_{s,\bnu}}(\TT^d) @> \hF_{\bnu}^n >>  H_{\La_{s,\bnu}}(\TT^d)\\
@ V \La_{s,\bnu} VV  @VV \La_{s,\bnu} V\\
L^2(\TT^d) @> \wQ_{\bnu, n} >> L^2(\TT^d)
\end{CD}
\end{displaymath}
suggests that we should instead study the operator
\begin{equation}\label{Pm semiclass}
\begin{split}
&\wP_{\bnu, n}
=\wQ_{\bnu,n}^*\wQ_{\bnu, n}= (\La_{s,\bnu}^{-1})^* \left[\left(\hF_{\bnu}^n\right)^* (\La_{s,\bnu}^*\La_{s,\bnu}) \hF_{\bnu}^n \right] \La_{s,\bnu}^{-1} \\
=& (\Op_\hb(\la_s)^{-1})^* \Big[(\Phi_\hb(1, S_{\bn_\bnu, n}))^* \big(\Op_\hb(\la_s)^*\Op_\hb(\la_s)\big) \Phi_\hb(1, S_{\bn_\bnu, n})\Big] \Op_\hb(\la_s)^{-1}.
\end{split}
\end{equation}

By the $\hb$-scaled symbol calculus (Theorem~\ref{Thm symbol cal hb}) 
and the $\hb$-scaled version of Egorov's theorem (Theorem~\ref{ThmEgorov2}), 
we have that $\wP_{\bnu, n}\in \OP_\hb S^0$, and it has a symbol 
of the form $\wtp_{\bn_\bnu, n} +\hb \wtr_{\bn_\bnu, n}$ 
given by Lemma~\ref{LPmu symbol2} below.  
Hence, by the $\hb$-scaled $L^2$-continuity theorem (Theorem \ref{ThmL2cont2}),
we may choose a sufficiently small $\ve_0>0$ such that 
\begin{equation*}
\|\wP_{\bnu, n}|L^2(\TT^d)\|
\le  \sup_{(x,\xi)\in T^*\TT^d} \wtp_{\bn_\bnu, n} (x,\xi)  +\ve_0
  +\hb C_{k_2}(\ve_0, \wtp_{\bn_\bnu,n}, \wtr_{\bn_\bnu, n}).
\end{equation*}
Moreover, by part (2) in Lemma~\ref{LPmu symbol2}, 
we get that there exists $L(n)>0$ independent of $\bnu$ such that 
$C_{k_2}(\ve_0, \wtp_{\bn_\bnu,n}, \wtr_{\bn_\bnu, n})\le L(n)$.
So we obtain
\begin{equation}\label{est Pnun0}
\|\wP_{\bnu, n}|L^2(\TT^d)\|
\le  \sup_{(x,\xi)\in T^*\TT^d} \wtp_{\bn_\bnu, n} (x,\xi) +\ve_0 +\hb L(n).
\end{equation}

By Sublemma~\ref{SLwtp le 1} (1) in the next section, 
we have that for any $n\in \NN$,
\begin{equation*}
\|\wP_{\bnu, n}|L^2(\TT^d)\|\le 2+L(n),
\end{equation*} 
and hence
\begin{equation}\label{est Fnun1}
\|\hF_\bnu^{n}|H_{\La_{s,\bnu}}(\TT^d)\|
= \sqrt{\|\wP_{\bnu, n}|L^2(\TT^d)\|}
\le \sqrt{2+L(n)}.
\end{equation}

Furthermore, by \eqref{est Pnun0} and our key lemma - Lemma~\ref{Lmain estmt} in the next section,
there are $n_0\in \NN$ and $\wtp_0<1$ such that for all $\bnu\in \ZZ^\ell\backslash\{\bzero\}$,
\begin{equation}\label{est Pnun1}
\begin{split}
\|\wP_{\bnu, n_0}|L^2(\TT^d)\| 
& \le  \sup_{(x,\xi)\in T^*\TT^d} |\wtp_{\bn_\bnu, n_0} (x,\xi)| + \ve_0 + \hb L(n_0) \\
& \le  \wtp_0+ \ve_0 + \dfrac{L(n_0)}{\ldb\bnu\rdb}.
\end{split}
\end{equation}
Note that we are allowed to make $\ve_0$ sufficiently small such that $\wtp_0+ \ve_0<1$
(although the function $L(\cdot)$ may be enlarged).
Then we choose $\nu_1>0$ such that 
\begin{equation}\label{def: rho2}
\rho_2:=\left(\wtp_0 + \ve_0 +\dfrac{L(n_0)}{\nu_1}\right)^{1/2n_0}<1.
\end{equation}
By \eqref{est Pnun1}, we have for all 
$\bnu\in \ZZ^\ell\backslash\{\bzero\}$ with $\lbnur=|\bnu|\ge \nu_1$,
\begin{equation}\label{est Fnun2}
\|\hF_\bnu^{n_0}|H_{\La_{s,\bnu}}(\TT^d)\|
= \sqrt{\|\wP_{\bnu, n_0}|L^2(\TT^d)\|}
\le  \sqrt{\wtp_0+\ve_0+\dfrac{L(n_0)}{\lbnur}}
\le \rho_2^{n_0}.
\end{equation}

Now for any $n\in \NN$, we write $n=kn_0+j$, where $k\in \NN_0$ and $0\le j<n_0$. Then by \eqref{est Fnun1} and \eqref{est Fnun2}, we have for all $\bnu\in \ZZ^\ell$ with $|\bnu|\ge\nu_1$, 
\begin{equation*}
\|\hF_\bnu^{n}|H_{\La_{s,\bnu}}(\TT^d)\|\le  \|\hF_\bnu^{n_0}|H_{\La_{s,\bnu}}(\TT^d)\|^k \|\hF_\bnu^{j}|H_{\La_{s,\bnu}}(\TT^d)\| 
\le  \rho_2^{kn_0} \sqrt{2+L(j)}\le C_2'\rho_2^n,
\end{equation*}
where we set
\begin{equation*}
C_2':=\max_{1\le j<n_0} \rho_2^{-j}\sqrt{2+L(j)}.
\end{equation*}

Switch back to the inner product $\lan \cdot, \cdot\ran_{s, \bnu}$, and recall the choice of $C_1$ in (\ref{choose: C1}). We take $C_2=C_1C_2'$, then 
for all $\bnu\in \ZZ^\ell$ with $|\bnu|\ge\nu_1$, 
\begin{equation*}
\|\hF_\bnu^{n}|H^s_\bnu(\TT^d)\|\le C_2\rho_2^n.
\end{equation*}
This completes the proof of Proposition~\ref{Pprop2}.
\end{proof}

\begin{lemma}\label{LPmu symbol2} 
Given $\bnu\in \ZZ^\ell$, let $\hb=1/\ldb\bnu\rdb$. 
For any $n\in \NN$, 
$\wP_{\bnu, n}\in \OP_\hb S^0$ has a symbol of the form $\wtp_{\bn_\bnu, n}+\hb \wtr_{\bn_\bnu, n}$, where $\wtp_{\bn_\bnu,n}\in S^0$ and $\wtr_{\bn_\bnu,n}\in S^{-1}$, such that
\begin{enumerate}
\item $\wtp_{\bn_\bnu, n}$ is positive, and is given by
\begin{equation}\label{fwtp nun}
\wtp_{\bn_\bnu, n} (x, \xi)= \sum_{x=T^ny} \CA_n(y) \left( \dfrac{g((D_yT^n)^t\xi +W_{\bn_\bnu, n}(y))}{g(\xi)}  \right)^{2s},
\end{equation}
where $\CA_n(y)$ is given in (\ref{fdef CAn}) and 
$W_{\bn, n}(y)$ is given in (\ref{fdef Wnun});
\item Let $C_{k_2}(\cdot, \cdot)$ be as introduced in Theorem \ref{ThmL2cont2}. There is $L(n)=L(n; d, T, \tau)$ independent of $\bnu$ such that
\begin{equation*}
C_{k_2}(\wtp_{\bn_\bnu, n}, \wtr_{\bn_\bnu, n})\le L(n).
\end{equation*}
\end{enumerate}
\end{lemma}

\begin{proof}
Recall that $\La_{s, \bnu}=\Op_\hb(\la_s)\in \OP_\hb S^s$. 
By Theorem \ref{Thm symbol cal hb},
$\La_{s,\bnu}^*\in \OP_\hb S^s$ has a symbol $\la_s \pmod{\hb S^{s-1}}$, and 
$\La_{s,\bnu}^{-1}$, $(\La_{s,\bnu}^{-1})^*\in \OP_\hb S^{-s}$ 
both have a symbol $\la_s^{-1} \pmod{\hb S^{-s-1}}$. 
Further, $\La_{s,\bnu}^*\La_{s,\bnu}\in \OP_\hb S^{2s}$ 
has a symbol $\la_s^2 \pmod{\hb S^{2s-1}}$.
By the $\hb$-scaled version of the Egorov's theorem (see Theorem~\ref{ThmEgorov2} and Remark~\ref{RmkEgorov}),
$\wtF_{\bnu, n}^* (\La_{s,\bnu}^*\La_{s,\bnu}) \wtF_{\bnu, n}\in \OP_\hb S^{2s}$ 
has a symbol 
\begin{eqnarray*}
\wa_n(x,\xi)&=&\sum_{ \substack{x=T^n y, \\ \xi=[(D_yT^n)^t]^{-1} [\eta-W_{\bn_\bnu, n}(y)]}} \la_s^2(y,\eta)\cdot 1^2\cdot |\det(D_yT^n)^t|^{-1} \pmod{\hb S^{2s-1}}\\
\\
&=& \sum_{x=T^n y} \dfrac{\la_s^2(y, (D_yT^n)^t\xi + W_{\bn_\bnu, n}(y))}{|\Jac(T^n)(y)|} \pmod{\hb S^{2s-1}}.
\end{eqnarray*}
By composition rule again, $\wP_{\bnu, n}\in \OP_\hb S^0$ has a symbol 
\begin{eqnarray*}
\wtp_{\bn_\bnu, n}(x,\xi)
&=&\sum_{x=T^n y} \dfrac{\la_s^2(y, (D_yT^n)^t\xi + W_{\bn_\bnu, n}(y))} 
     {|\Jac(T^n)(y)|} \dfrac{1}{\la_s^2(x,\xi)} \pmod{\hb S^{-1}}\\
&=& \sum_{x=T^n y} \dfrac{1}{|\Jac(T^n)(y)|} \dfrac{h(y)}{h(x)} \dfrac{g((D_yT^n)^t\xi + W_{\bn_\bnu, n}(y))^{2s}}{g(\xi)^{2s}} \pmod{\hb S^{-1}}\\
&=& \sum_{x=T^n y} \CA_n(y) \left(\dfrac{g((D_yT^n)^t\xi + W_{\bn_\bnu, n}(y))}{g(\xi)}\right)^{2s} \pmod{\hb S^{-1}}.
\end{eqnarray*}
This finishes the proof of the first part.

Since all the above modulo terms are calculated from the symbol $\la_s$ and the phase $S_{\bn_\bnu, n}$, which depends on $\bn_\bnu$ but not $\lbnur$, we can write the full symbol of $\wP_{\bnu, n}$ by $\wtp_{\bn_\bnu, n}+\hb \wtr_{\bn_\bnu, n}$ for some $\wtr_{\bn_\bnu, n}\in S^{-1}$. 

By Theorem \ref{ThmL2cont2}, $C_{k_2}(\wtp_{\bn_\bnu,n}, \wtr_{\bn_\bnu, n})$ is bounded by a constant which only depends on the $\CN_{k_2}$-seminorms of $\wtp_{\bn_\bnu,n}$ and $\wtr_{\bn_\bnu,n}$. To prove the second part of this lemma, it is sufficient to show that  the $\CN_{k_2}$-seminorms of $\wtp_{\bn_\bnu,n}$ and $\wtr_{\bn_\bnu,n}$ are bounded by a term that does not depend on $\bnu$.
Indeed, by the formula of $\wP_{\bnu, n}$ in \eqref{Pm semiclass}, the $\hb$-scaled symbol calculus and Egorov's theorem, we find that the $\CN_{k_2}$-seminorm of $\wtp_{\bn_\bnu,n}$ depends on 
$\CN_{k_2}$-seminorm of $\la_s$ and $\CN_{k_2+2}$-seminorm of $S_{\bn_\bnu, n}$; 
while the $\CN_{k_2}$-seminorm of $\wtr_{\bn_\bnu,n}$ depends on $\CN_{k_2+2}$-seminorm of $\la_s$ and $\CN_{k_2+4}$-seminorm of $S_{\bn_\bnu, n}$.
Clearly the $\CN_{k_2}$-and $\CN_{k_2+2}$-seminorms of $\la_s$ 
are independent of $\bnu$. 
Note that $S_{\bn, n}$ is given by \eqref{def phase nu n}, and then we have
\begin{equation*}\label{semi norm Snun}
\begin{split}
\CN_{k_2+2}(S_{\bn_\bnu, n})\le &\CN_{k_2+2}(T^nx\cdot \xi) +\left|\bn_\bnu \right| \CN_{k_2+2}\left(\left|\sum_{k=0}^{n-1} \tau(T^kx)\right|\right) \\
\le &\CN_{k_2+2}(T^nx\cdot \xi) + \CN_{k_2+2}\left(\left|\sum_{k=0}^{n-1} \tau(T^kx)\right|\right),
\end{split}
\end{equation*} 
which implies that the $\CN_{k_2+2}$-seminorm of $S_{\bn_\bnu, n}$ is independent of $\bnu$.
Similarly, $\CN_{k_2+4}$-seminorm of $S_{\bn_\bnu, n}$ is also independent of $\bnu$.
This is what we need.
\end{proof}

\section{Estimates of $\wtp_{\bn_\bnu, n}$: Proof of Lemma~\ref{Lmain estmt}}

\subsection{Lemma~\ref{Lmain estmt} and Its Proof}

The estimates given in Lemma~\ref{Lmain estmt} in this section 
is the most important step to prove Proposition~\ref{Pprop2}.

\begin{lemma}\label{Lmain estmt} 
If $\tau(x)$ is not an essential coboundary over $T$, 
then there exists $n_0\in \NN$ such that 
\begin{equation}\label{eq pn le 1}
\wtp_0:
=\sup_{\bnu\in \ZZ^\ell\setminus \{\bzero\}}\sup_{(x,\xi)\in T^*\TT^d} 
   \wtp_{\bn_\bnu, n_0} (x,\xi)
<1.
\end{equation}
\end{lemma}

\begin{proof} 
Given $\bn\in \SSS^{\ell-1}$ and $x\in \TT^d$, 
we consider the affine map $\CF_{\bn, x}: \RR^d \to \RR^d$ given by
\begin{equation}\label{fdef CF}
\CF_{\bn, x}(\xi) =(D_xT)^t \xi + (D_x \tau)^t \bn \ \ \text{for any} \ \xi\in \RR^d,
\end{equation}
and the $n$-th iterates 
\begin{equation}\label{fdef CFn}
\CF_{\bn, x}^n(\xi)
=\prod_{k=0}^{n-1} \CF_{\bn, T^k x} (\xi)=(D_xT^n)^t\xi + W_{\bn, n}(x) \ \ \text{for any}\  n\in \NN,
\end{equation}
where $W_{\bn, n}(x)=W_n(x)\bn$, and $W_n(x)$ is given by \eqref{fdef Wnun}. 
Conventionally, we set $W_{\bn, 0}(x)=\bzero$ and $\CF_{\bn, x}^0=\id$.
We also define
\begin{equation}\label{fwtp nun1}
\wtp_{\bn, n}(x,\xi) =\sum_{x=T^n y} 
\CA_n(y)\left[ \dfrac{g\left(\CF^n_{\bn, y}(\xi) \right)}{g(\xi)}\right]^{2s},
\end{equation}
which extends the formula of $\wtp_{\bn_\bnu, n}(x,\xi)$ given by \eqref{fwtp nun} to all $\bn\in \SSS^{\ell-1}$. 
For any fixed $n\in \NN$, it is easy to check that the function $(\bn, x, \xi)\mapsto \wtp_{\bn, n}(x, \xi)$ is of class $C^\infty$.
Further, we set 
\begin{equation}\label{fdef wtp bn n}
\wtp_{\bn, n}:=\sup_{(x,\xi)\in T^*\TT^d} \wtp_{\bn, n}(x,\xi).
\end{equation}
The properties of $\wtp_{\bn, n}$ will be given by Sublemma~\ref{property wtp bn n} below.

Recall that $R$ is given by \eqref{fdef R}.
By Sublemma~\ref{SLwtp <1} below, for any $\bn\in \SSS^{\ell-1}$, there exists $n_0(\bn)\in \NN$ such that
for any $(x,\xi)\in T^*\TT^d$, 
we have $|\CF^{n_0(\bn)}_{\bn, y}(\xi)|>2 R$ for some $y\in T^{-n_0(\bn)}(x)$.
Then by \eqref{fdef CFn} and \eqref{fdef Wnun}, we have for any $\bn'\in \SSS^{\ell-1}$, 
\begin{eqnarray*}
|\CF^{n_0(\bn)}_{\bn', y}(\xi)| 
&\ge & |\CF^{n_0(\bn)}_{\bn, y}(\xi)|-|\CF^{n_0(\bn)}_{\bn', y}(\xi)-\CF^{n_0(\bn)}_{\bn, y}(\xi)|\\
&\ge & 2 R- \left|W_{n_0(\bn)}(y)\right| |\bn'-\bn| \\
&\ge & 2 R- \|DT\|\;\|D\tau\|\; n_0(\bn)\; |\bn'-\bn|.
\end{eqnarray*}
Hence  there is $\ve(\bn)>0$ such that $|\CF^{n_0(\bn)}_{\bn', y}(\xi)|> R$ whenever $|\bn'-\bn|<\ve(\bn)$. 
By Sublemma~\ref{SLwtp le 1} (2), we get $\wtp_{\bn', n_0(\bn)}(x, \xi)<1$. 
Since $\wtp_{\bn', n_0(\bn)}(x, \xi)$ is continuous with $x$ and $\xi$,
we get that for any compact set $\CU\subset T^*\TT^d$, 
$\max_{(x, \xi)\in \CU} \wtp_{\bn', n_0(\bn)}(x, \xi) <1$.
Together with Sublemma~\ref{SLwtp le 1} (3), we obtain
\begin{equation*}
\begin{split}
\wtp_{\bn', n_0(\bn)}=& \max\left\{\ \sup_{(x,\xi)\in \CU_R} \wtp_{\bn', n_0(\bn)}(x,\xi), \ \sup_{(x,\xi)\in T^*\TT^d\setminus \CU_R} \wtp_{\bn', n_0(\bn)}(x, \xi)\ \right\} \\
\le & \max\left\{\ \max_{(x, \xi)\in \CU_R} \wtp_{\bn', n_0(\bn)}(x, \xi),\  \left(\frac{\ga+1}{2}\right)^{2s} \right\} <1,
\end{split}
\end{equation*}
where $\CU_R:=\{(x,\xi): |\xi|\le R\}$ is a compact subdomain in $T^*\TT^d$, and $\ga$ is given by \eqref{fdef R}.
Moreover, by Sublemma~\ref{property wtp bn n} (2),
$\wtp_{\bn', n}\le  \wtp_{\bn', n_0(\bn)}<1$ for any $n\ge n_0(\bn).$

To sum up, for any $\bn\in \SSS^{\ell-1}$, there are $n_0(\bn)\in \NN$ and $\ve(\bn)>0$ such that $\wtp_{\bn', n}<1$ 
for all $\bn'\in B(\bn, \ve(\bn))$ and $n\ge n_0(\bn)$, where $B(\bn, \ve(\bn))$ denotes the open ball in $\SSS^{\ell-1}$ 
with center at $\bn$ and of radius $\ve(\bn)$. Since $\SSS^{\ell-1}$ is compact, there are $\bn_1, \bn_2, \dots, \bn_k\in \SSS^{\ell-1}$ 
such that the finite collection of open balls $\{B(\bn_j, \ve(\bn_j))\}_{1\le j\le k}$ covers ${\SSS}^{\ell-1}$. 
Therefore, if we set 
\begin{equation*}
n_0=\max\{n_0(\bn_1), \dots, n_0(\bn_k)\},
\end{equation*}
then $\wtp_{\bn, n_0}<1$ for all $\bn\in \SSS^{\ell-1}$. By Sublemma~\ref{property wtp bn n} (3), we know that the function $\bn\mapsto \wtp_{\bn, n_0}$ is continuous, then 
$\sup\limits_{\bn\in \SSS^{\ell-1}} \wtp_{\bn, n_0}=\max\limits_{\bn\in \SSS^{\ell-1}} \wtp_{\bn, n_0} <1$, 
from which (\ref{eq pn le 1}) follows. 
\end{proof}

\subsection{Sublemmas and Proofs}

Recall that $R$ and $\ga$ are given by \eqref{fdef R} and 
\eqref{fdef gamma} respectively. 

\begin{sublemma}\label{SLwtp le 1} 
Let $\bn\in \SSS^{\ell-1}$.  Then
\begin{enumerate}
\item $\wtp_{\bn, n}(x, \xi)\le 1$ for all $n\in \NN$ and $(x, \xi)\in T^*\TT^d$;
\item $\wtp_{\bn, n}(x, \xi)< 1$ if and only if there is $y\in T^{-n} x$ such that $|\CF^n_{\bn, y}(\xi)|>R$;
\item $\wtp_{\bn, n}(x, \xi)\le \left(\frac{\ga+1}{2}\right)^{2s}<1$ for all $n\in \NN$ and $(x,\xi)\in T^*\TT^d$ with $|\xi|>R$.
\end{enumerate}
\end{sublemma}

\begin{proof} 
The key observation is the following: for any $n\in \NN$, $y\in \TT^d$ and $\xi\in \RR^d$,
\begin{equation}\label{CF xi}
\left|\CF^n_{\bn, y} (\xi)\right|>\frac{\ga+1}{2} |\xi| \ \ \text{if}\  |\xi|>R. 
\end{equation}
Indeed, by the choice of $R$ in (\ref{fdef R}), we have 
$|\xi|>R>\dfrac{2\|D\tau\|}{\ga-1}$. So by \eqref{fdef CF},
\begin{eqnarray*}
\left|\CF_{\bn, y} (\xi)\right|
\ge |(D_y T)^t\xi|-\left| (D_y\tau)^t \bn\right| 
\ge \ga|\xi|-\|D\tau\|
\ge \ga|\xi|-\dfrac{\ga-1}{2}|\xi|=\dfrac{\ga+1}{2}|\xi|.
\end{eqnarray*}
Hence by induction, we have for all $n\ge 1$, 
\begin{equation*}
\left|\CF_{\bn, y}^n(\xi)\right|\ge  \left(\dfrac{\ga+1}{2}\right)^n |\xi|\ge \dfrac{\ga+1}{2} |\xi|. 
\end{equation*}
Consequently, for any $n\in \NN$, $y\in \TT^d$ and $\xi\in \RR^d$,
\begin{itemize}
\item if $\left|\CF_{\bn, y}^n(\xi)\right|\le R$, then $\left|\CF_{\bn, y}^k(\xi)\right|\le R$ for all $0\le k\le n$, and in particular, we must have $|\xi|\le R$;
\item if $\left|\CF_{\bn, y}^n(\xi)\right|> R$, then $\left|\CF_{\bn, y}^n(\xi)\right|> |\xi|$ no matter whether $|\xi|>R$ or not.
\end{itemize}
By the definition of $g(\xi)$ given by (\ref{fdef g}), the quotient 
\begin{equation*}
\left[\dfrac{g\left(\CF^n_{\bn, y}(\xi) \right)}{g(\xi)}\right]^{2s} 
\begin{cases}  = 1
 & \ \text{if}\  \left|\CF^n_{\bn, y}(\xi)\right|\le R; \\
  \\ <1 & \ \text{otherwise}. 
\end{cases}
\end{equation*}
In either case, we always get 
\begin{equation}\label{quotient le 1}
0<\left[\dfrac{g\left(\CF^n_{\bn, y}(\xi) \right)}{g(\xi)}\right]^{2s} \le 1.
\end{equation}
Recall that
$\sum\limits_{x=T^n y} \CA_n(y)=1$, 
where $\CA_n(y)$ is positive and defined by \eqref{fdef CAn}. 
Therefore, for any $n\in \NN$ and $(x,\xi)\in T^*\TT^d$,
\begin{equation*}
\wtp_{\bn, n}(x,\xi)=\sum_{x=T^n y} \CA_n(y)\left[ \dfrac{g\left(\CF^n_{\bn, y}(\xi) \right)}{g(\xi)}\right]^{2s}\le \sum_{x=T^n y} \CA_n(y)=1.
\end{equation*}
Clearly, we have that $\wtp_{\bn,n}(x,\xi)<1$ if and only if  $\left|\CF^n_{\bn, y}(\xi)\right|> R$ for some $y\in T^{-n} x$.
Moreover, if $|\xi|>R$, then by \eqref{CF xi}, we have that
\begin{equation*}
\left[\dfrac{g\left(\CF^n_{\bn, y}(\xi) \right)}{g(\xi)}\right]^{2s} \le \left(\frac{\ga+1}{2}\right)^{2s} \ \ \text{for all}\ y\in T^{-n}x,
\end{equation*}
and hence 
$\wtp_{\bn, n}(x,\xi)\le \left(\frac{\ga+1}{2}\right)^{2s}$.
\end{proof}

\begin{sublemma}\label{property wtp bn n} 
Let $\wtp_{\bn, n}$ be defined as in \eqref{fdef wtp bn n}. 
\begin{enumerate}
\item For any $\bn\in \SSS^{\ell-1}$ and $n\in \NN$, we have that $0<\wtp_{\bn, n}\le 1$;
\item For any $\bn\in \SSS^{\ell-1}$, the sequence $\{\wtp_{\bn, n}\}_{n\in\NN}$ is non-increasing;
\item For any $n\in \NN$, the function $\bn\mapsto \wtp_{\bn, n}$ is continuous.
\end{enumerate}
\end{sublemma}

\begin{proof} 
As shown by Sublemma~\ref{SLwtp le 1} (1), $0<\wtp_{\bn, n}(x, \xi)\le 1$ for any $\bn\in \SSS^{\ell-1}$, $n\in\NN$ and $(x,\xi)\in T^*\TT^d$, 
and thus $0<\wtp_{\bn, n}\le 1$.

Let $\bn\in \SSS^{\ell-1}$ be fixed. For any $(x,\xi)\in T^*\TT^d$ and $n, m\in \NN$, by \eqref{quotient le 1} we have 
\begin{equation*}
\begin{split}
\wtp_{\bn, n+m}(x, \xi)
=&\sum_{x=T^{n+m} y} \CA_{n+m}(y)\left[ \dfrac{g\left(\CF^{n+m}_{\bn, y}(\xi) \right)}{g(\xi)}\right]^{2s} \\
=&\sum_{x=T^nz} \sum_{z=T^m y} \CA_n(z)\CA_{m}(y)\left[ \dfrac{g\left(\CF^{m}_{\bn, y}(\xi) \right)}{g(\xi)}\right]^{2s} \left[ \dfrac{g\left(\CF_{\bn, z}^n\left(\CF^{m}_{\bn, y}(\xi) \right)\right)}{g\left(\CF^{m}_{\bn, y}(\xi) \right)}\right]^{2s}\\
\le & \sum_{x=T^nz}\CA_n(z)  \sum_{z=T^m y}  \CA_{m}(y)\left[ \dfrac{g\left(\CF^{m}_{\bn, y}(\xi) \right)}{g(\xi)}\right]^{2s}\\
=& \sum_{x=T^nz} \CA_n(z) \ \wtp_{\bn, m}(z, \xi) \le \wtp_{\bn, m} \sum_{x=T^nz} \CA_n(z)= \wtp_{\bn, m}.
\end{split}
\end{equation*}
Hence, $\wtp_{\bn, n+m}=\sup\limits_{(x,\xi)\in T^*\TT^d}  \wtp_{\bn, n+m}(x, \xi)\le \wtp_{\bn, m}$. This proves that  the sequence $\{\wtp_{\bn, n}\}_{n\in\NN}$ is non-increasing.

Let $n\in \NN$ be fixed. To show that the function $\bn\mapsto \wtp_{\bn, n}$ is continuous, it suffices to show that the family 
\begin{equation*}
\{\bn\mapsto \wtp_{\bn, n}(x, \xi):\ (x,\xi)\in T^*\TT^d\} \subset C^0(\SSS^{\ell-1}) 
\end{equation*}
is uniformly bounded and equicontinuous. The uniform boundedness is already given by Sublemma~\ref{SLwtp le 1} (1). The equicontinuity follows from that
\begin{equation*}
\begin{split}
 &\sup_{(x,\xi)\in T^*\TT^d} \left|\frac{\p}{\p \bn} \wtp_{\bn, n}(x,\xi)\right|\\
=&\sup_{(x,\xi)\in T^*\TT^d} \left|\sum_{x=T^n y} \CA_n(y) \dfrac{2s \left[ g\left(\CF^n_{\bn, y}(\xi) \right)\right]^{2s-1}}{\left[g(\xi)\right]^{2s}}[Dg\left(\CF^n_{\bn, y}(\xi)\right) ]^t W_n(y) \right| \\
\le & 2|s| \sup_{(x,\xi)\in T^*\TT^d} \sum_{x=T^n y} \CA_n(y) \left[ \dfrac{g\left(\CF^n_{\bn, y}(\xi) \right)}{g(\xi)}\right]^{2s} 
\dfrac{\|Dg\|\cdot n\|DT\|\|D\tau\|}{g\left(\CF^n_{\bn, y}(\xi) \right)} \\
\le & 2n |s|  \|Dg\| \|DT\| \|D\tau\|\ \wtp_{\bn, n}\le 2n |s|  \|Dg\| \|DT\| \|D\tau\|<\infty.
\end{split}
\end{equation*}
Here we have used \eqref{fdef CFn}, \eqref{fwtp nun1}, \eqref{fdef Wnun} and the following properties of $g(\xi)$ (see (\ref{fdef g})):
$g(\xi)\ge 1$ for any $\xi\in \RR^d$ and $\|Dg\|=\sup_{\xi\in \RR^d} |Dg(\xi)|<\infty$.
\end{proof}

\begin{sublemma}\label{SLwtp <1} 
Suppose $\tau(x)$ is not an essential coboundary over $T$.
For any $\bn\in \SSS^{\ell-1}$, there is $n_0(\bn)\in \NN$ such that 
for any $(x,\xi)\in T^*\TT^d$, 
$|\CF^{n_0(\bn)}_{\bn, y}(\xi)|>2 R$ for some $y\in T^{-n_0(\bn)}(x)$.
\end{sublemma}

\begin{proof} 
Let us argue by contradiction.  If this sublemma does not hold for some $\bn^*\in \SSS^{\ell-1}$, then for any $n\in \NN$, there is $(x_n, \xi_n)\in T^*\TT^d$ such that $|\CF^n_{\bn^*, y}(\xi_n)|\le 2 R$ for any $y\in T^{-n}(x_n)$. In fact, we further have 
\begin{equation}\label{CF k xi}
|\CF^k_{\bn^*, y}(\xi_n)|\le 2 R \quad\ \text{for any}\ y\in T^{-n}(x_n) \ 
\text{and} \ 0\le k\le n,
\end{equation}
since otherwise if $|\CF^k_{\bn^*, y}(\xi_n)|> 2 R$ for some $0\le k<n$, then by \eqref{CF xi}, $|\CF^n_{\bn^*, y}(\xi_n)|=|\CF^{n-k}_{\bn^*, T^ky}(\CF^k_{\bn^*, y}(\xi_n))|\ge \frac{\ga+1}{2}|\CF^k_{\bn^*, y}(\xi_n)|>2R$. Note that in particular, $|\xi_n|\le 2R$.\\

For $y\in T^{-n}(x_n)$ and $0\le k\le n$, denote 
\begin{equation}\label{fdef wW}
\wW_{\bn^*, k}(y) =[(D_yT^k)^t]^{-1} W_{\bn^*, k}(y) 
= \sum_{j=0}^{k-1} 
[(D_{T^j y} T^{k-j})^t]^{-1} [D_{T^j y}\tau]^t \bn^*.
\end{equation}
Using (\ref{fdef CFn}) and then \eqref{CF k xi}, we can write
\begin{equation*}
\left|(D_yT^k)^t\left(\xi_n + \wW_{\bn^*, k}(y) \right)\right|
=\left|(D_yT^k)^t\xi_n + W_{\bn^*, k}(y)\right|
=\left|\CF^k_{\bn^*, y}(\xi_n)\right|
\le 2R. 
\end{equation*}
By (\ref{fdef gamma}), we have 
\begin{equation}\label{Convergence ga}
\left|\xi_n + \wW_{\bn^*, k}(y) \right|\le \dfrac{2R}{\ga^k},
\qquad \text{for any} \ y\in T^{-n}(x_n), \ 0\le k\le n.
\end{equation}

We would like to rewrite $\wW_{\bn^*, k}(y)$ in terms of $x_n$ as follows. 
Suppose the degree of the expanding endomorphism $T:\TT^d\to \TT^d$ is $N$.
We denote 
\begin{equation*}
\Si_N^n=\{\bi=(i_1, i_2, \dots, i_n): \ i_j=0, 1, \dots, N-1\}, 
\ \ 1\le n\le \infty.
\end{equation*} 
Let $T_0^{-1}, T_1^{-1}, \dots, T_{N-1}^{-1}$ be the inverse branches of $T$.
Given $x\in \TT^d$ and $\bi\in \Si_N^n$,  
we denote
$T_{\bi}^{-j}x=T_{i_j}^{-1}\dots T_{i_1}^{-1}x$, 
which is well-defined whenever $0\le j\le n\le \infty$ and $j$ is finite.
We then define 
\begin{equation}\label{fdef Vk}     
V_{\bn^*, k}(\bi, x):= \sum_{j=1}^{k} 
D_x\left[\tau(T^{-j}_\bi(x)) \cdot \bn^*\right] 
= \sum_{j=1}^{k} 
[(D_{T^{-j}_\bi x} T^j)^t]^{-1} (D_{T^{-j}_\bi x} \tau)^t \bn^*.
\end{equation}
for any $1\le k\le n\le \infty$ and $k$ is finite.

Note that 
\begin{equation}\label{fVnu conv}
\sum_{j=m}^\infty 
\left| [(D_{T^{-j}_\bi x} T^j)^t]^{-1} [D_{T^{-j}_\bi x} \tau]^t \bn^* \right|
\le \|D\tau\| \sum_{j=m}^\infty \ga^{-j}
\to 0  \quad \text{as} \  m\to \infty, 
\end{equation}
and the convergence is uniform.  
That is, the sequence $\{V_{\bn^*, k}(\bi, x)\}_1^{\infty}$ is uniform Cauchy.
Hence $V_{\bn^*, \infty}(\bi, x)$ is well-defined as in \eqref{fdef Vk} 
for all $x\in\TT^d$ and $\bi\in\Si_N^\infty$.
Denote $V_{\bn^*}(\bi, x)=V_{\bn^*,\infty}(\bi, x)$.  We have 
\begin{equation}\label{fVnu conv1}
\lim_{k\to\infty}V_{\bn^*, k}(\bi, x) = V_{\bn^*}(\bi, x) 
\qquad  \text{for any} \ x\in\TT^d, \ \bi\in\Si_N^\infty,
\end{equation}
and the convergence is uniform.  
Moreover, by \eqref{fVnu conv}, for any $n>0$, 
\begin{equation}\label{fVnu conv2}
\left| V_{\bn^*, n}(\bi, x) - V_{\bn^*}(\bi, x)\right|
\le  \frac{\|D\tau\| }{\ga^{n}(1-\ga)}.
\end{equation}

Comparing \eqref{fdef wW} and \eqref{fdef Vk}, we see 
$\wW_{\bn^*, n}(y)=V_{\bn^*, n}(\bi, x)$ whenever $y=T_\bi^{-n}(x)$.
Hence by (\ref{Convergence ga}) and (\ref{fVnu conv2}), for any $n\in \NN$,
\begin{equation}\label{fVnu conv3}
\begin{split}
\left|\xi_n + V_{\bn^*}(\bi, x_n) \right|
\le &\left|\xi_n + V_{\bn^*, n}(\bi, x_n) \right|
 +  \left|V_{\bn^*, n}(\bi, x_n)-  V_{\bn^*}(\bi, x_n) \right|  \\
\le &\dfrac{2R}{\ga^n}+\dfrac{\|D\tau\|}{\ga^n(1-\ga)}.
\end{split}
\end{equation}

Since the sequence $\{(x_n, \xi_n)\}$ lies in the compact subdomain 
$\CU_{2R}:=\{(x,\xi): |\xi|\le 2R\}$ of $T^*\TT^d$,
there is an accumulation point $(x^*, \xi^*)$.
Choosing subsequences if necessary, we take $n\to\infty$ and 
obtain from \eqref{fVnu conv3} that
$V_{\bn^*}(\bi, x^*)=-\xi^*$,
regardless of the choice for $\bi\in \Si_N^\infty$.

For any $x\in \TT^d$, take $w\in\{0, 1, \dots, N-1\}$ such that $x=T_w^{-1}(Tx)$.
For any $\bi\in \Si_N^\infty$, we can directly check \eqref{fdef Vk} (when $k=\infty$)
to get
\begin{equation}\label{fVnu formula}
(D_{x}T)^t V_{\bn^*}(w\bi, Tx)
=V_{\bn^*}(\bi, x) + (D_{x}\tau)^t\bn^*.
\end{equation}

By Claim~\ref{Claim1} below, we know that 
$V_{\bn^*}(\bi, x)$ is independent of $\bi$ for any $x\in \TT^d$.  
Hence, we can define a function $V_{\bn^*}: \TT^d\to \RR^d$ by
\begin{equation*}
V_{\bn^*}(x)=V_{\bn^*}(\bi, x), \ \ \text{for any}\ \bi\in \Si_N^\infty,
\end{equation*}
and thus \eqref{fVnu formula} is rewritten as 
\begin{equation}\label{cohom derivative}
(D_xT)^tV_{\bn^*}(Tx)
=V_{\bn^*}(x)+(D_{x}\tau)^t\bn^*.
\end{equation}

By Claim~\ref{Claim2} below, which asserts that the 1-form on $\TT^d$ given by $V_{\bn^*}(x)\cdot dx$ is exact, there is a potential function 
$u$ such that $\nabla_x u=V_{\bn^*}(x)$. Alternatively, we can define the function $u:\TT^d\to \RR$ by
\begin{equation*}
u(x)=\int_{\Ga_{\bzero, x}} V_{\bn^*}(z)\cdot dz, \ \ x\in \TT^d,
\end{equation*}
where $\Ga_{\bzero, x}$ is any smooth path in $\TT^d$ from $\bzero=(0, 0, \dots, 0)$ to $x$. 

On both sides of (\ref{cohom derivative}), we replace $x$ by $tx$, take the dot product with $x$ and 
integrate with respect to $t$ from 0 to 1, then we get
\begin{equation*}
\int_{\Ga^1_{T\bzero, Tx}} V_{\bn^*}(z)\cdot dz = \int_{\Ga^0_{\bzero, x}} V_{\bn^*}(z)\cdot dz + \int_{\Ga^0_{\bzero, x}} (D_{z}\tau)^t\bn^*\cdot dz,
\end{equation*}
where $\Ga^0_{\bzero, x}:=\{tx: 0\le t\le 1\}$, and $\Ga^1_{T\bzero, Tx}:=\{T(tx): 0\le t\le 1\}$. In other words, we have 
\begin{equation*}
u(Tx)-u(T\bzero)=u(x)-u(\bzero)+\bn^*\cdot \tau(x) -\bn^*\cdot \tau(\bzero).
\end{equation*}
Note that $u(\bzero)=0$. Let $c=\bn^*\cdot \tau(\bzero)-u(T\bzero)$, then we get 
\begin{equation*}
\bn^*\cdot \tau(x)=c-u(x)+u(Tx),
\end{equation*}
which contradicts to the fact that $\tau(x)$ is not 
an essential coboundary over $T$. 
\end{proof}

\begin{remark}
We would like to mention that Faure and Weich 
constructed a function similar to \eqref{fdef Vk} in \cite{MR3719542}, Proposition 4.9, 
which plays an important role in the study of asymptotic spectral gap for open partially expanding systems.
\end{remark}

\begin{claim}\label{Claim1}
Suppose that there exists a point $(x^*, \xi^*)$ such that $V_{\bn^*}(\bi, x^*)=-\xi^*$,
regardless of the choice for $\bi\in \Si_N^\infty$.
Then for any $x\in\TT^d$, $V_{\bn^*}(\bi, x)$ is independent of $\bi$, that is,
$V_{\bn^*}(\bi, x)=V_{\bn^*}(\bi', x)$ for all $\bi, \bi'\in\Si_N^\infty$.
\end{claim}

\begin{proof}
Taking $x=T_w^{-1} x^*$ in \eqref{fVnu formula} for some $w\in \{0, 1, \dots, N-1\}$,  we get
\begin{equation*}
V_{\bn^*}(\bi, T^{-1}_w x^*) =-(D_{T^{-1}_w {x^*}}T)^t \xi^* 
-[D_{T^{-1}_w {x^*}}\tau]^t\bn^*.
\end{equation*}
The right hand side is independent of $\bi$, and hence
$V_{\bn^*}(\bi, T^{-1}_w x^*)=V_{\bn^*}(\bzero, T^{-1}_w x^*)$, where
$\bzero=(0, 0, \dots)\in \Si_N^\infty$. 
 
Inductively, one can show that $V_{\bn^*}(\bi, x)=V_{\bn^*}(\bzero, x)$ for all 
$x\in \bigcup_{n=1}^\infty T^{-n}(x^*)$ and thus for all $x\in \TT^d$, 
since the set $\bigcup_{n=1}^\infty T^{-n}x^*$ is dense in $\TT^d$.
\end{proof}

\begin{claim}\label{Claim2}
The 1-form on $\TT^d$ given by 
\begin{equation*}
V_{\bn^*}(x)\cdot dx=V_{\bn^*}^1(x) dx_1+\dots+V_{\bn^*}^d(x) dx_d
\end{equation*} 
is exact.
\end{claim}

\begin{proof}
We first show that $V_{\bn^*}(x)\cdot dx$ is a closed 1-form, which is equivalent to showing that for any $x\in \TT^d$,
\begin{equation}\label{proof closed}
\frac{\p}{\p x_i} V_{\bn^*}^j(x)=\frac{\p}{\p x_j} V_{\bn^*}^i(x), \ \ 1\le i\le j\le d.
\end{equation}
Indeed, by \eqref{fVnu conv1} and Claim 1, $V_{\bn^*, k}^j(\bi, x)$ converges uniformly to $V_{\bn^*}^j(\bi, x)=V_{\bn^*}^j(x)$ as $k\to\infty$.
By similar calculation as in \eqref{fVnu conv}, we have that $\frac{\p}{\p x_i} V_{\bn^*, k}^j(\bi, x)$ converges uniformly as $k\to\infty$, 
and hence $\frac{\p}{\p x_i} V_{\bn^*}^j(x)=\lim\limits_{k\to \infty} \frac{\p}{\p x_i} V_{\bn^*, k}^j(\bi, x)$.
We see from \eqref{fdef Vk} that for each $k\in \NN$ and any $\bi\in\Si^\infty_N$,
the 1-form $V_{\bn^*, k}(\bi, x)\cdot dx=d\left(\sum_{j=1}^k \tau(T^{-j}_\bi(x))\cdot \bn^*\right)$ is exact and hence closed. Thus, 
\begin{equation*}
\frac{\p}{\p x_i} V_{\bn^*, k}^j(\bi, x)=\frac{\p}{\p x_j} V_{\bn^*, k}^i(\bi, x), \ \ 1\le i\le j\le d,
\end{equation*}
from which \eqref{proof closed} follows by taking $k\to\infty$.\\

Now we show that $V_{\bn^*}(x)\cdot dx$ is exact.
Since $V_{\bn^*}(x)\cdot dx$ is closed, it is sufficient to prove that
for any $x=(x_1, x_2, \dots, x_d)\in \TT^d$ and $1\le k\le d$,
\begin{equation}\label{cohom integral}
\int_0^1 V^k_{\bn^*}(x_1, \dots, x_{k-1}, t, x_{k+1}, \dots, x_d)\ dt=0.
\end{equation}
To see this, 
by (\ref{fdef Vk}) and Claim~\ref{Claim1}, we rewrite for arbitrary $M\in \NN$, 
\begin{align*}
&V_{\bn^*}(x) 
= \frac{1}{N^{M}} \sum_{\bi\in \Si_N^M} V_{\bn^*}(\bi 0^\infty, x) \\
=& \frac{1}{N^{M}} \sum_{\bi\in \Si_N^M} V_{\bn^*, M}(\bi 0^\infty, x) 
+ \frac{1}{N^{M}} 
  \sum_{\bi\in \Si_N^M} [V_{\bn^*}(\bi 0^\infty, x)-V_{\bn^*, M}(\bi 0^\infty, x)]\\
=& \frac{1}{N^{M}} \sum_{j=1}^M \sum_{\bi\in \Si_N^j} 
 D_x\left[ \tau(T^{-1}_{i_j}\dots T^{-1}_{i_1}(x))\cdot \bn^* \right] 
 + \frac{1}{N^{M}} 
  \sum_{\bi\in \Si_N^M} [V_{\bn^*}(\bi 0^\infty, x)-V_{\bn^*, M}(\bi 0^\infty, x)] \\
=&: I_{\bn^*}(x) + J_{\bn^*}(x).
\end{align*}
Here we denote 
$\bi 0^\infty=(i_1, i_2, \dots, i_n, 0, 0, \dots)\in \Si_N^\infty$
for any $\bi=(i_1, i_2, \dots, i_n)\in \Si_N^M$.

On one hand, let $I_{\bn^*}^k$ be the $k$-th component of $I_{\bn^*}$, then
\begin{align*}
  & \int_0^1 I_{\bn^*}^k(x_1, \dots, x_{k-1}, t, x_{k+1}, \dots, x_d) dt \\
=& \frac{1}{N^{M}} \sum_{j=1}^M   \sum_{\bi\in \Si_N^j} 
\int_0^1 \frac{\p}{\p x_k} 
\left[ 
\tau(T^{-1}_{i_j}\dots T^{-1}_{i_1}(x_1, \dots, x_{k-1}, t, x_{k+1}, \dots, x_d ))\cdot \bn^*
\right] dt \\
=&   \frac{1}{N^{M}} \sum_{j=1}^M 
\ \bn^*\cdot \sum_{\bi\in \Si_N^j} \Bigl[
\tau(T^{-1}_{i_j}\dots T^{-1}_{i_1}(x_1, \dots, x_{k-1}, 
          1, x_{k+1}, \dots, x_d )) \\
&\qquad\qquad\qquad\quad
-\tau(T^{-1}_{i_j}\dots T^{-1}_{i_1}(x_1, \dots, x_{k-1}, 
      0, x_{k+1}, \dots, x_d ))
     \Bigr] \\
=&   \frac{1}{N^{M}} \sum_{j=1}^M 
\ \bn^*\cdot \Bigl[  \sum_{\bi\in \Si_N^j} 
\tau(T^{-1}_{i_j}\dots T^{-1}_{i_1}(x_1, \dots, x_{k-1}, 
          1, x_{k+1}, \dots, x_d )) \\
&\qquad\qquad\quad 
-  \sum_{\bi\in \Si_N^j} \tau(T^{-1}_{i_j}\dots T^{-1}_{i_1}(x_1, \dots, x_{k-1}, 
      0, x_{k+1}, \dots, x_d ))
     \Bigr]  =0.    
\end{align*}

The last term must vanish since $\{T^{-1}_{i_j} \dots T^{-1}_{i_1}(x_1, \dots, x_{k-1}, 0, x_{k+1}, \dots, x_d ):  \bi\in \Si_N^j\}$
and $\{T^{-1}_{i_j}\dots T^{-1}_{i_1}(x_1, \dots, x_{k-1}, 1, x_{k+1}, \dots, x_d ): \bi\in \Si_N^j\}$ are just two representations for
the set of all $j$-th pre-images of the point
$(x_1, \dots, x_{k-1}, 0, x_{k+1}, \dots, x_d )=(x_1, \dots, x_{k-1}, 1, x_{k+1}, \dots, x_d )$ in $\TT^d$.

On the other hand, by \eqref{fVnu conv} and \eqref{fVnu conv1}, the convergence 
$V_{\bn^*, M}(\bi 0^\infty, x)\to V_{\bn^*}(\bi 0^\infty, x)$ is uniform in $\bi$ and $x$
as $M\to \infty$.
By choosing $M$ large enough, the integral of the $k$-th component of $J_{\bn^*}(x_1, \dots, x_{k-1}, t, x_{k+1}, \dots, x_d)$ with respect to $t$
from 0 to 1 is arbitrary small and hence 0. It follows that \eqref{cohom integral} holds.
\end{proof}

\bibliography{decaybib}{}
\bibliographystyle{plain}

\end{document}